\documentclass[a4paper,11pt]{article}

\usepackage{amsmath}
\usepackage{amssymb}
\usepackage{amsthm}
\usepackage{refcount}
\usepackage[normalem]{ulem} 

\usepackage[dvips,pdftex]{color}
\usepackage{pgf,tikz}
\usetikzlibrary{shapes,automata,positioning,arrows}
\usepackage{graphicx}
\usepackage{setspace}
\usepackage[version=3]{mhchem}
\usepackage{chemfig} 
\usepackage[sort&compress,comma,square,numbers]{natbib}
\usepackage{booktabs}

\usepackage{enumerate,hyperref}
\usepackage{xypic}

\usepackage{url}

\usepackage{mathabx}
 
\newcommand{\R}{\mathbb{R}}
\newcommand{\N}{{\mathbb N}}
\newcommand{\Z}{{\mathbb Z}}

\newcommand{\mC}{\mathcal{C}}
\newcommand{\mP}{\mathcal{P}}
\newcommand{\mS}{\mathcal{S}}
\newcommand{\mR}{\mathcal{R}}
\newcommand{\mN}{\mathcal{N}}
\newcommand{\mG}{\mathcal{G}}

\DeclareMathOperator*{\im}{im}

\DeclareMathOperator*{\rank}{rank}

\newcommand{\rev}{\ce{<=>[]}}

\renewcommand{\k}{{\kappa}}
\newcommand{\st}{\mid}

\newtheorem{lemma}{Lemma}
\newtheorem{proposition}{Proposition}
\newtheorem{theorem}{Theorem}
\newtheorem{corollary}{Corollary}

\theoremstyle{definition}
\newtheorem{definition}{Definition}
\newtheorem{example}{Example}

\newcommand\red[1]{\textcolor{red}{#1}}

\begin{document}

\title{Node Balanced Steady States: Unifying and Generalizing Complex and Detailed Balanced Steady States}

\author{Elisenda Feliu$^{1,3}$, Daniele Cappelletti$^2$, Carsten Wiuf$^1$}
\date{\today}

\footnotetext[1]{Department of Mathematical Sciences, University of Copenhagen, Universitetsparken 5, 2100 Copenhagen, Denmark.}
\footnotetext[2]{Department of Mathematics, University of Wisconsin-Madison, Van Vleck Hall 480 Lincoln Drive, Madison, Wi 53706, USA.}
\footnotetext[3]{Corresponding author: efeliu@math.ku.dk}

 \tikzset{every node/.style={auto}}
 \tikzset{every state/.style={rectangle, minimum size=0pt, draw=none, font=\normalsize}}
 \tikzset{bend angle=15}

\maketitle

  \begin{abstract}
We introduce a unifying and generalizing framework for complex and detailed balanced steady states in chemical reaction network theory. To this end, we generalize the  graph commonly used to represent a reaction network. Specifically, we introduce a graph, called a \emph{reaction graph}, that has one edge for each reaction but potentially multiple nodes for each  complex. A special class of steady states, called \emph{node balanced steady states}, is naturally associated with such a reaction graph. We show that complex and detailed balanced steady states are special cases of node balanced steady states by choosing appropriate reaction graphs. Further, we show that node balanced steady states have properties analogous to complex balanced steady states, such as  uniqueness and asymptotical stability in each stoichiometric compatibility class.
Moreover, we associate an integer, called the \emph{deficiency}, to a reaction graph that gives the number of independent relations in the reaction  rate constants that need to be satisfied for a positive node balanced steady state to exist.

The set of reaction graphs (modulo isomorphism) is equipped with a partial order that has the complex balanced reaction graph as minimal element. We relate this order to the deficiency and to the set of  reaction rate constants for which a positive node balanced steady state exists. 
\end{abstract}


 \section{Introduction}
Complex balanced steady states of a chemical reaction network are perhaps the most well-described class of  steady states in chemical reaction network theory. Horn and Jackson built a theory for positive complex balanced steady states and showed  that they are unique and asymptotically stable relatively to the linear invariant subspace they belong to \cite{hornjackson}.  Around the same time, Feinberg studied 
 a structural network property, 
 called the \emph{deficiency}, and derived  parameter-independent theorems concerning the existence of complex balanced steady states, based on the deficiency \cite{feinberg1972,feinberg-def0}.  
 
The graphical structure of a reaction network also plays an integral part of the present work. In fact we will not stick to a single graphical representation of a reaction network but to a collection of graphical representations, and build a theory that extends the classical theory of complex and detailed balanced steady states. In this theory detailed and complex balanced steady states arise as particular examples of the same phenomenon. 
 
 The standard graphical representation of a reaction network is a digraph where the nodes are the complexes of the network and the directed edges are the reactions, as in the example below with two species. This representation appears so natural that a reaction  network  might be defined directly as a digraph with nodes labeled by linear combinations of the species \cite{Dickenstein:2011p1112} as follows:
\begin{equation}\label{eq:graph1}
\textrm{
\schemestart
$X_1+X_2$  \arrow(1--2){->[{\small $r_1$}]}[-45] $X_1+2X_2$
 \arrow(--3){->[][{\small $r_2$}]}[180]  $X_2$ \arrow(--@1){->[{\small $r_3$}]}[45]  
 \arrow(@2--4){<=>[{\small $r_4$}][{\small $r_5$}]} $2X_1$.
\schemestop
}
\end{equation}
In addition, we also consider digraphs where the same complex in different reactions might or might not be represented by the same node. These graphs, called \emph{reaction graphs} (Definition~\ref{def:reaction-graph}), can be obtained from the standard graph by duplication of nodes.
As an example, consider the digraph 
\begin{equation}\label{eq:graph2}
\textrm{
\schemestart
$X_1+X_2$  \arrow(1--2){->[{\small $r_1$}]}[-45] $X_1+2X_2$
 \arrow(--3){->[][{\small $r_2$}]}[180]  $X_2$ \arrow(--@1){->[{\small $r_3$}]}[45]  
\schemestop
\qquad
\schemestart
 $X_1+2X_2$  \arrow(1--2){<=>[{\small $r_4$}][{\small $r_5$}]} $2X_1$, \schemestop}
\end{equation}
where the node with label $X_1+2X_2$ is duplicated, such that the digraph  \eqref{eq:graph1} is split into two components.
The  digraph \eqref{eq:graph1} is obtained by  collapsing the two nodes for $X_1+2X_2$, that is, by reversing the duplication step. 

We associate a \emph{novel} type of steady states, called 
\emph{node balanced steady states} with a given reaction graph (Definition~\ref{def:node-balanced}). To set the idea, recall that a complex balanced steady state is an equilibrium point such that, for any complex $y$, the sum of the reaction flows out of $y$ equals the sum of the reaction flows going into $y$. To illustrate this, consider the digraph \eqref{eq:graph1} with mass-action kinetics. A complex balanced steady state  $x=(x_1,x_2)$ fulfills
\begin{equation}\label{eq:intro1} \k_1 x_1x_2=\k_3 x_2,\qquad \k_3 x_2= \k_2x_1x_2^2,\qquad    \k_2x_1x_2^2+\k_4 x_1x_2^2=\k_1 x_1x_2 +\k_5 x_1^2,
\end{equation}
where   $\k_i$ is the  reaction rate constant of the $i$-th reaction.

Analogously, we define a node balanced steady state with respect to a given reaction graph as a steady state fulfilling 
 equations similar to \eqref{eq:intro1}, derived from the particular reaction graph.
Thus, under mass-action kinetics,  
a node balanced steady state of the  digraph \eqref{eq:graph2} fulfills 
\begin{equation}\label{eq:intro2}  \k_1 x_1x_2=\k_3 x_2,\qquad \k_3 x_2= \k_2x_1x_2^2,\qquad    \k_2x_1x_2^2 =\k_1 x_1x_2,\qquad  \k_4 x_1x_2^2=\k_5 x_1^2.
\end{equation}
The equations in \eqref{eq:intro1} can be obtained by adding the last two equations in \eqref{eq:intro2}. A node balanced steady state of the digraph  \eqref{eq:graph2} is  therefore in particular  complex balanced.

In this context   complex balanced steady states are  instances of node balanced steady states for a particular choice of reaction graph. As we will see, detailed balanced steady states are  also node balanced steady states for a specific reaction graph. 
It is therefore not surprising that node balanced steady states satisfy  properties analogous to complex (and detailed) balanced steady states. In fact, we  show that many classical results carry over to node balanced steady states and might be defined in terms of properties of  reaction graphs, rather than reaction networks.
Particularly, 
if there is one positive node balanced steady state, then all  steady states are node balanced, and   if this is so, there is one positive node balanced steady state in each stoichiometric compatibility class. Furthermore, this steady state is asymptotically stable relatively to the class (Theorem~\ref{thm:complex}). 
Additionally, we give algebraic conditions  on the reaction  rate constants for which node balanced steady states exist with respect to a given reaction graph  (Theorem~\ref{prop:k}).  There are as many algebraic relations as the \emph{deficiency}
of the reaction graph. This also generalizes  known results for complex balanced steady states \cite{Craciun-Sturmfels}.

We define a natural partial order on the set  of reaction graphs. The standard graphical representation of  a reaction network, as in \eqref{eq:graph1}, is the unique minimal element.
Intuitively, 
a reaction graph $G$ is smaller than, or included in,  
another reaction graph $G'$, $G\preceq G'$, if $G$ can be obtained from $G'$ by collapsing some of the nodes of $G'$.  We will show that if $G\preceq G'$, then the deficiency of $G$ is smaller than or equal to that of $G'$ (Proposition~\ref{prop:deficiency}). 
As an example, the reaction graph in \eqref{eq:graph1} is smaller than that of \eqref{eq:graph2}, but their deficiencies  are the same.

Horn and Jackson showed that conditions for complex and detailed balanced steady states 
could be stated in terms of symmetry conditions on the reaction rates \cite{hornjackson}. They also  speculated that perhaps there were other classes of networks for which the steady states fulfilled similar symmetry conditions. 
We show that node balanced steady states might indeed be defined in terms of symmetry conditions, similar to those of Horn and Jackson.

The motivation for this work comes from the desire to build a general unifying framework for complex and detailed balanced steady states. 
However,  as a consequence of our results, we are additionally able to state sufficient conditions on the reaction rate constants of a network such that a part of the network is at steady state whenever the whole system is at  steady state. This is a relevant question in the context of biological modeling, since it is often the case that  only subnetworks of a system are studied. It is therefore natural to wonder whether the small network is at steady state  when the whole network is, and \emph{vice versa}.

The structure of the paper is as follows. In the next section we introduce reaction networks and reaction graphs, together with basic properties of reaction graphs. In Section \ref{sec:node}, we define node balanced steady states and discuss their properties. After that, in Section \ref{sec:HJ} and \ref{sec:part}, we discuss the symmetry conditions of Horn and Jackson, and the relationship between a part and the whole of a reaction network. Finally, in Section \ref{sec:proof}, we provide  proofs of the main theorems on node balanced steady states.

\section{Reaction networks and reaction graphs}

We let $\R^n_{\geq 0}$ and $\R^n_{> 0}$ denote the nonnegative and positive orthants of $\R^n$,  respectively. If $x\in\R^n_{>0}$ ($\R^n_{\ge 0}$), then we say that $x$ is positive (nonnegative). Similarly, $\Z_{\geq0}$ denotes the nonnegative integers. If $v_1,\ldots,v_k\in\R^n$ are vectors, then $\langle v_1,\ldots,v_k\rangle$ denotes the linear subspace generated by the vectors. 

\subsection{Reaction networks}
This section  introduces  reaction networks and  their associated ODE systems \cite{gunawardena-notes,feinbergnotes}. 
A \emph{reaction network} (or  simply a network) is a triplet $\mN= ( \mS, \mC, \mR)$
where  $\mS$, $\mC \subseteq \Z_{\geq 0}^{\mS}$ and 
$\mR\subseteq \mC \times \mC$
are  finite sets, called respectively 
the \emph{species, complex } and \emph{reaction} set.  
We implicitly assume  the sets are  numbered and let
$$\mS = \{X_1,\dots,X_n\}, \quad \mC=\{y_1,\dots,y_m\}, \quad\textrm{and}\quad \mR=\{r_1,\dots,r_p\}, $$
such that  $n,m,$ and $p$ are their respective  cardinalities. 
 Hence $\Z_{\geq 0}^{\mS} \cong \Z_{\geq 0}^n$, and a  complex can be identified with a  linear combination of species
$y=(\lambda_1,\dots,\lambda_n)=\sum_{i=1}^n \lambda_i X_i$. 
We further assume that  $(y,y)\not\in\mR$, any $y\in\mC$ is in at least one reaction and any $S\in\mS$ is in at least one complex. 
In that case, $\mC$ and $\mS$ can be found from $\mR$, and  $\mN$ is said to be \emph{generated} from $\mR$.

An element $r=(y,y')$ of $\mR$ is written  as $r\colon y\rightarrow y'$ or just $y\rightarrow y'$.
A reaction $y\rightarrow y'\in \mR$ is \emph{reversible} if $y'\rightarrow y\in \mR$ as well. If this is not the case, then the reaction is \emph{irreversible}. 
A pair of reactions $y\rightarrow y'$ and $y'\rightarrow y$ is called a reversible reaction pair and denoted by $y\rev y'$. A reaction network is \emph{reversible} if all  reactions of the network are reversible.

The \emph{stoichiometric  matrix} $N\in \R^{n\times p}$ is the matrix with  $j$-th column $N_j=y'-y$ where $r_j\colon y\rightarrow y'$.
The columns of $N$ generate the so-called \emph{stoichiometric subspace} $S$ of $\R^n$. We define  $s=\rank(N)=\dim(S)$.

We let $x_i$ denote the concentration of species $X_i$. 
A kinetics $v$ for a reaction network is a $\mC^1$-function from $\R^n_{\geq 0}$ to $\R^p_{\geq 0}$ such that $v(\R^n_{>0})\subseteq \R^p_{>0}$.
The $j$-th coordinate, $v_j$, is called the \emph{rate function} of $r_j$.
The main example of kinetics is \emph{mass-action kinetics}, where
\begin{equation}\label{eq:mass-action}
v_{j}(x)= \k_j x^{y} = \k_j \prod_{i=1}^n x_i^{\lambda_i}, \qquad \textrm{for }\quad r_j\colon y\rightarrow y',\quad \text{and}\quad y=(\lambda_1,\ldots,\lambda_n),
\end{equation}
and  $\k_j>0$ denotes the \emph{reaction rate constant} of $r_j$. By convention, $0^0=1$.   Whenever the numbering of $\mR$ is irrelevant, we write $v_{y\rightarrow y'}$ and $\k_{y\rightarrow y'}$, instead of $v_j$ and $\k_j$, respectively.

Given a kinetics $v$, the evolution of the species concentrations over time is modeled by a system of ODEs,
\begin{equation}\label{eq:ode} \frac{dx}{dt} = Nv(x),\qquad x\in \R^n_{\geq 0}.
\end{equation}
For reasonable kinetics, including mass-action kinetics, the solution of \eqref{eq:ode}  is in the positive (nonnegative) orthant for all positive times in the interval of definition, if the initial condition is \cite{Sontag:2001}.
Furthermore, the solution is confined to one of the 
 nonnegative polytopes known as the \emph{stoichiometric compatibility classes}
$$ \mP_{x_0}= (x_0+S) \cap \R^n_{\geq 0},$$
where $x_0\in\R_{\ge 0}$ is the initial condition.

The \emph{steady states} of \eqref{eq:ode} are  the nonnegative solutions to the equation
$Nv(x)=0.$
For mass-action kinetics, this equation becomes:
\begin{equation}\label{eq:steady-states}
\sum_{r_j\colon y\rightarrow y'}  \k_j x^{y} (y'-y) =0,\qquad x\in \R^n_{\geq 0}.
\end{equation}

\refstepcounter{example}\label{ref:exampleHorn}
\newcounter{mycounterHorn}
\renewcommand{\themycounterHorn}{\getrefnumber{ref:exampleHorn}\,(part\,\Alph{mycounterHorn})}
\newtheorem{myexampleHorn}[mycounterHorn]{Example}

\begin{myexampleHorn}\label{ex:horn}
We use a variant of a reaction network in  \cite[equation (7.3)]{hornjackson} and \cite[Example 2.3]{Feinberg:1989vg} as a `running' example throughout the paper.
The set $\mR$ of reactions  consists of
\begin{align*}
r_1\colon && 3X_1 & \rightarrow X_1+2X_2 &  r_2 \colon &&  X_1+2X_2 & \rightarrow 3X_2 & r_3\colon &&   3X_2 & \rightarrow 2X_1+X_2 \\
r_4\colon &&  2X_1+X_2 & \rightarrow 3X_1  & r_5\colon &&  3X_1 & \rightarrow 3X_2 &   r_6\colon && 3X_2 & \rightarrow 3X_1,
\end{align*}
with $\mC=\{3X_1,X_1+2X_2,3X_2,2X_1+X_2\}$ and $\mS=\{X_1,X_2\}$.
There is one reversible reaction pair, $r_5$ and $r_6$, 
 and four irreversible reactions.
The stoichiometric matrix  is 
$$ N=\left(\begin{array}{rrrrrr}
-2 & -1 & 2 & -1 & -3 & 3 \\ 
2 & 1 & -2 & 1 & 3 & -3
\end{array}\right),$$
with rank $s=1$. Thus, the stoichiometric compatibility classes $(x_0+\langle(-1,1)\rangle) \cap \R^2_{\geq 0}$ are one-dimensional.
Mass-action kinetics implies  $v(x)= ( 
\k_1 x_1^3, \k_2 x_1x_2^2,   \k_3 x_2^3, \k_4 x_1^2x_2, \k_5 x_1^3, \k_6 x_2^3)$.
\end{myexampleHorn}

\subsection{Reaction graphs}
In this subsection we introduce novel graphical representations of a reaction network, the main object of this work. 
 
\begin{definition}\label{def:reaction-graph}
Let $\mN= ( \mS, \mC, \mR)$  be a reaction network. 
A \emph{reaction graph} associated with $\mN$ is a node labeled digraph $G=(V_G,E_G,Y)$  with no isolated nodes, where 
$V_G=\{1,\dots,m_G\}$ and
\begin{itemize}
\item[(i)] $Y$ is a surjective labeling  of the  node set with values in $\mC$:
$$ V_G \xrightarrow{Y} \mC,\qquad\textrm{such that}\qquad  i \mapsto Y_i,$$
\item[(ii)] $Y$ induces a bijection $R$ between $E_G$ and $\mR$ as follows:
$$E_G \xrightarrow{R} \mR,\qquad\textrm{such that}\qquad  i\rightarrow j  \quad \mapsto \quad R_{i\rightarrow j} 
= Y_i \rightarrow Y_j. $$
\end{itemize}
We say that a reaction graph is \emph{weakly reversible} if all connected components of $G$ are strongly connected, that is, there exists a directed path between two nodes, whenever there is a directed path between the nodes in opposite direction.
\end{definition}

In what follows, we generically denote a complex of a network as $y_i$ and a label of a node in a reaction graph  as $Y_i$. For a reaction graph $G'$, we denote objects related to it  with $'$.
Note that two nodes can have the same label in $\mC$, but  each reaction of $\mN$ corresponds to exactly one edge of $G$.
For two reaction graphs  $G,G'$  of $\mN$, the bijections $R, R'$ induce a correspondence $R^{-1}\circ R'$ between  $E_{G'}$ and $E_G$:
An edge $i'\rightarrow j'$ of $G'$ \emph{corresponds} to the edge $i\rightarrow j$ of $G$, if 
the two edges  map  to the same reaction of $\mN$.

Any permutation of the node set $V_G$ of a reaction graph gives an identical reaction graph, except for the numbering of the nodes. In the Introduction, `reaction graph' was used in the sense of `up to a numbering of the nodes' without mentioning it explicitly.

\begin{figure}[!t]
\begin{center}
\includegraphics[scale=0.8]{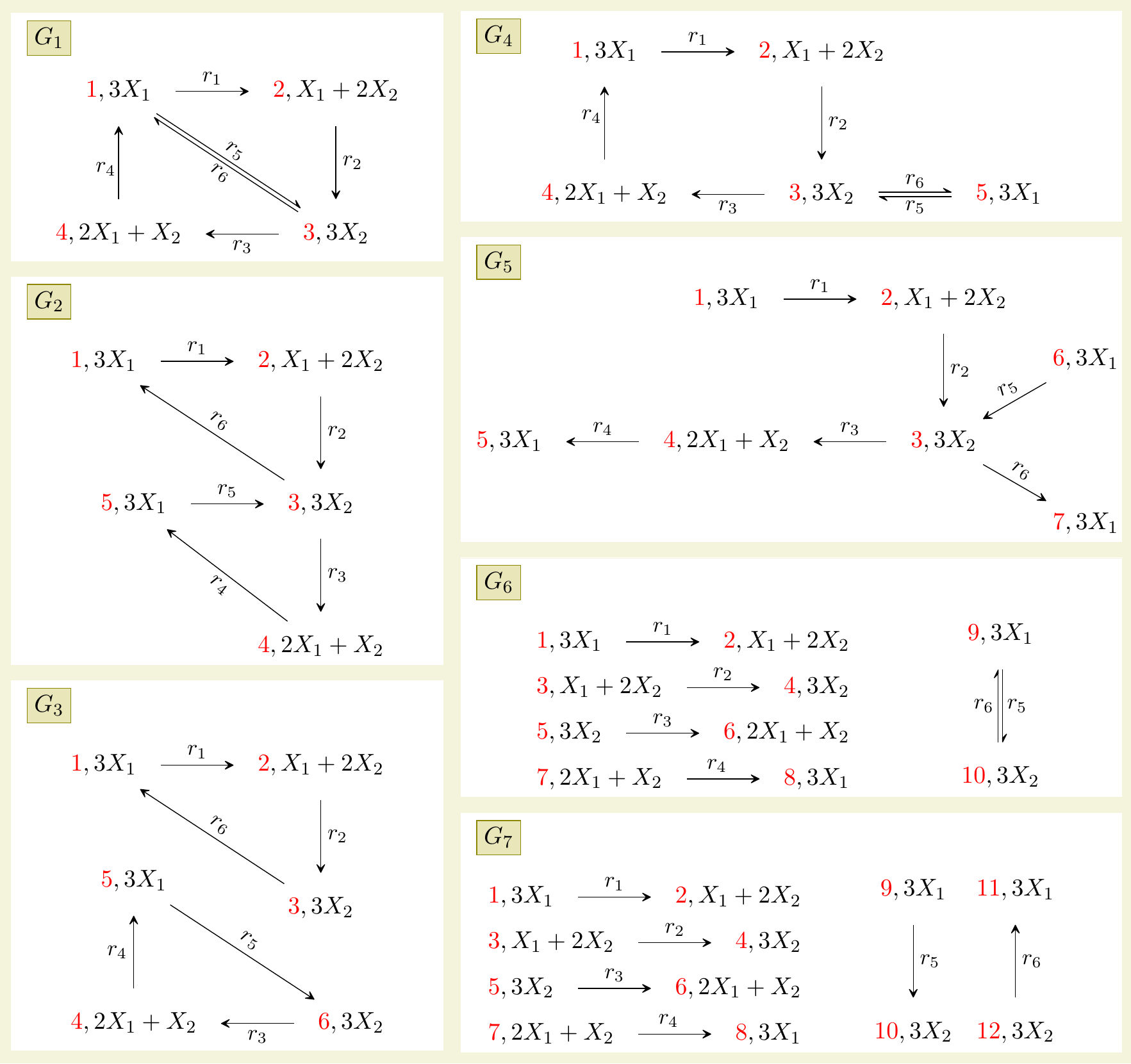}
\end{center}
\caption{Reaction graphs for the network in Example~\ref{ex:horn}. Reaction labels $r_j$ are added for convenience. } \label{fig:horn}
\end{figure}

\begin{myexampleHorn}\label{ex:horngraphs}
The following digraph and  labeling function
\begin{center}
\begin{minipage}[h]{0.45\textwidth}
\begin{center}
\schemestart
$1$  \arrow(1--2){->} $2$
 \arrow(--3){->}[270] $3$ \arrow(--4){->}[180] $4$
 \arrow(@4--@1){->} 
 \arrow(@1--@3){<=>}
\schemestop
\end{center}
\end{minipage}
\begin{minipage}[h]{0.45\textwidth}
\begin{align*}
 Y_1 & = 3X_1, & Y_2 &  = X_1+2X_2, \\  Y_3 & = 3X_2, &  Y_4 & = 2X_1+X_2,
\end{align*}
\end{minipage}
\end{center}
define a reaction graph for the   reaction network in Example~\ref{ex:horn}. This reaction graph is weakly reversible. In general we draw the node labels next to the nodes and add  labels ($r_j$) to the edges, as in the  digraph  $G_1$ of Figure~\ref{fig:horn}.
The edge labels are redundant information,
 because the bijection $R$ of a reaction graph  $(V_G,E_G,Y)$ is explicitly determined.
 However, it makes  comparison between reaction graphs easier.

Figure~\ref{fig:horn} shows seven reaction graphs for the   reaction network in Example~\ref{ex:horn} that will be used  to illustrate results.
The reaction graphs except for $G_5$, $G_6$ and $G_7$ are weakly reversible.  
 \end{myexampleHorn}

The following definition names some special reaction graphs. Recall that we assume that numberings of the reaction and complex sets are given.

\begin{definition}
Let $\mN= ( \mS, \mC, \mR)$  be a reaction network.
\begin{itemize}
\item  A \emph{complex reaction graph} $G$  is   a reaction graph with $m$ nodes such that the labeling $Y$ is a bijection between $V_G$ and $\mC$.  
The \emph{canonical  complex reaction graph}  fulfills
$$Y_j= y_j,\qquad j=1,\dots,m.$$
\item 
A \emph{detailed reaction graph} is  a reaction graph with one connected component per  reversible reaction pair and one per  irreversible reaction.
\item A \emph{split reaction graph}  is   a reaction graph with $2p$ nodes and one connected component per reaction.
The \emph{canonical split reaction graph} 
fulfills
 $$ Y_{2j-1} =y,\qquad     Y_{2j} = y',\qquad \textrm{if}\quad r_j \colon y\rightarrow y',\quad  \textrm{for all}\ j=1,\dots,p.$$
\item We say that \emph{$\mN$ is weakly reversible} if any complex reaction graph is weakly reversible.
\end{itemize} 
\end{definition}

\begin{myexampleHorn}
Consider  
the reaction graphs in Figure~\ref{fig:horn}. The reaction graph $G_1$ is a complex reaction graph, $G_6$ is a detailed reaction graph  and $G_7$ is a split reaction graph (in fact, the canonical split reaction graph). The network is weakly reversible.
\end{myexampleHorn}

 Intuitively, any reaction graph $G$ is obtained by collapsing or joining nodes of a split reaction graph with the same labels.  Oppositely, we can view a reaction graph as a graph where some  nodes of a complex reaction graph  are duplicated.
 Only nodes that are source or target nodes of multiple edges can be duplicated.
 For example, if node 1 of the reaction graph $G_1$ in Figure~\ref{fig:horn} is duplicated, then 
 the reaction graphs $G_4$ and $G_2$ are obtained. Collapsing the node pairs $(9,11)$, and $(10,12)$ of the split reaction graph $G_7$ yields $G_6$.  
  Node duplication  does not determine the reaction graph uniquely. 
  In contrast, by collapsing  node pairs 
  the reaction graph is uniquely determined. This is formalized in the next subsection.

\subsection{Morphisms of reaction graphs and a partial order} 

In this subsection we consider 
the family of reaction graphs associated with a reaction network $\mN$ and show that   the set of equivalence classes of reaction graphs forms a lattice.

\begin{definition}\label{def:morphismgraph}
Let $G,G'$ be two reaction graphs associated with a reaction network $\mN$ and let $Y,Y'$ be their labeling functions, respectively.
A \emph{morphism of reaction graphs} from $G$ to $G'$  is a map between the node sets 
$$ \varphi\colon V_{G} \rightarrow V_{G'},$$
such that 
\begin{enumerate}[(i)]
\item $\varphi(i)\rightarrow \varphi(j)$ is an edge of $G'$ if $i\rightarrow j$ is an edge of $G$.
\item $Y'_{\varphi(i)} = Y_{i}$ for all $i = 1,\dots,m_{G}$ (equivalently, $Y=Y'\circ \varphi$). 
\end{enumerate}
An \emph{isomorphism} of reaction graphs is a morphism that has an inverse morphism. It  is simply a permutation of the nodes of the graph. 
\end{definition}

By definition, a morphism of reaction graphs is in particular a digraph morphism. Isomorphism of reaction graphs is an equivalence relation, which allows us to speak about the \emph{equivalence class} $[G]$ of a reaction graph $G$. In this sense, two reaction graphs are \emph{equivalent} if they are isomorphic. The  complex reaction graphs form an equivalence class $\mathcal{G}_c$ and the split reaction graphs form another class $\mathcal{G}_{sp}$, with representatives given by  the  canonical reaction graphs.
 In the same way, the detailed reaction graphs  also form an equivalence class. 
An equivalence class of reaction graphs can be depicted by omitting the numbering of the node set, as we did  in \eqref{eq:graph1} and \eqref{eq:graph2} in the Introduction. We will use this representation in some examples to introduce  an equivalence class without specifying a representative.

 \begin{lemma}\label{lem:surjective}
  Any morphism of reaction graphs $\varphi\colon V_G\to V_{G'}$ is surjective.
 \end{lemma}
\begin{proof}
 Definition~\ref{def:morphismgraph} implies that 
$R_{i\rightarrow j} = Y_i\rightarrow Y_j = Y'_{\varphi(i)} \rightarrow Y'_{\varphi(j)}  = R'_{\varphi(i)\rightarrow \varphi(j)}$.
Therefore, $(R')^{-1}\circ R (i\rightarrow j)= \varphi(i)\rightarrow \varphi(j)$. Since $(R')^{-1}\circ R$ is a bijection between $E_G$ and $E_{G'}$ and  reaction graphs have no isolated nodes,   all nodes of $G'$ are in the image of $\varphi$, that is, $\varphi$ is surjective.
\end{proof}
 
Let $G_{sp}=(V_{sp},E_{sp},Y_{sp})$ denote  the  canonical split reaction graph. 

\begin{definition}\label{def:partition}
A partition $\mathcal{P}$ of $V_{G_{sp}}=\{1,\dots,2p\}$ 
is called \emph{admissible} if 
$$Y_{sp,i}=Y_{sp,j} \quad\text{ for all }i,j \text{ in the same subset }P\in \mathcal{P}.$$ 
Given an admissible partition $\mP= \{P_1,\dots,P_q\}$, the associated reaction graph $G_\mP=(V_{\mP} ,E_{\mP} ,Y_\mP)$  is defined as
\begin{align*}
V_{\mP} & = \{1,\dots,q\}, \qquad (\textrm{so, }m_{G_\mP}=q),\\
E_{\mP} &= \{ k \rightarrow k' \st   i\rightarrow j \in E_{G_{sp}},\, i \in P_k, j\in P_{k'} \}, \\
(Y_\mP)_k &= Y_{sp,i} \qquad \textrm{ if}\qquad i\in P_k.
\end{align*}
\end{definition}
\smallskip
Note that the associated reaction graph depends on the numbering of the sets in the partition, which we implicitly give when writing $\mP= \{P_1,\dots,P_q\}$.
The map $Y_\mP$ is well defined because the partition is admissible. Moreover, $E_{\mP}$ is in one-to-one correspondence with $\mR$ because an edge $k\rightarrow k'$  arises from a unique choice of $i,j$. Otherwise,  there would be two edges in $G_{sp}$ corresponding to the same reaction, since the partition is admissible. 

\begin{lemma}\label{lem:equiv}
\begin{enumerate}[(i)]
\item
For any reaction graph $G$ there exists an admissible partition $\mP=\{P_1,\ldots,P_q\}$ with $q=m_G$ elements,  such that   $G=G_\mP$.
\item
Two reaction graphs $G,G'$ with respective admissible partitions $\mP=\{P_1,\ldots,P_q\}$, $\mP'=\{P_1',\ldots,P_{q'}'\}$, as in (i), are equivalent if and only if 
$\mP=\mP'$. 
\end{enumerate}
\end{lemma}
\begin{proof} 
(i) We construct the partition $\mP$, such that there is one set $P_k$ for each node $k\in V_G$. 
For each edge $i'\rightarrow j'$ of $G_{sp}$, 
let $i\rightarrow j$ be the edge of $G$ corresponding to the same reaction. 
Then  by definition $i'\in P_i$ and $j'\in P_j$.  Each node of $G_{sp}$ is  the source or target of exactly  one edge.  Therefore, all nodes are assigned a unique set $P_k$, and each set $P_k$ has at least one element.  Thus $P_1,\dots,P_{m_G}$ is a numbered partition of  
$\{1,\ldots,2p\}$, which  further is admissible.  It is straightforward to check that $G=G_\mP$.

(ii) Consider admissible partitions $\mP,\mP'$ and numberings of the subsets of the partitions such that $G=G_\mP$ and $G'=G_{\mP'}$. The isomorphism between $G$ and $G'$ translates into a permutation of the numbering of the subsets of the partitions, and thus $\mP=\mP'$.
\end{proof}

We conclude from Lemma~\ref{lem:equiv} that an equivalence class of reaction graphs can be identified with an  admissible partition $\mP$ of the set $\{1,\dots,2p\}$. This class is denoted by $\mathcal{G}_\mP$.
 Using  this, we can define a partial order on the set of  equivalence classes.

\begin{definition}
Let
$\mP$ and $\mP'$ be two admissible partitions. We say that $\mP$ is a \emph{refinement} of $\mP'$, and write $\mP\leq \mP'$, if for each $P\in \mP$
there exists $P'\in \mP'$  such that $P\subseteq P'$.

Given two equivalence classes  
$\mathcal{G}_\mP, \mathcal{G}_{\mP'}$, we define $\mathcal{G}_{\mP'}\preceq\mathcal{G}_{\mP}$ if  $\mP\leq \mP'$. 
\end{definition}

If $[G']\preceq [G]$, then we  write $G'\preceq G$ for convenience. We say that $G'$ is included in  $G$, or that $G$ includes $G'$.
Note that the  objects are reversed in $[G']\preceq [G]$ and $\mP\leq \mP'$.

 The admissible partition for 
 $\mG_{sp}$ has $2p$ subset each with one element: $ \{i\}$, for $i=1,\dots,2p$. 
 The admissible partition $\mP_c$ defining $\mG_c$ has $m$ elements, one for each complex:
$$ P_{c,y}= \big\{ j\in \{1,\ldots,2p\} 
 \st Y_{sp,j} = y   \big\},\qquad y\in \mC.$$ 
Any admissible partition is a refinement  of $\mP_c$, and hence an admissible partition is a union of partitions, one for each subset of $\mP_c$.  Further, the partition for the split reaction graphs is a refinement of 
 any other admissible partition. Hence,  for any reaction graph $G$, it holds
 \begin{equation}
  \label{eq:complex-small}
 \mG_c\preceq [G] \preceq \mG_{sp}.
\end{equation}

\medskip
The set of admissible partitions inherits a lattice structure from the lattice structure of the set of partitions in general. Recall that a lattice is a set  with a partial order such that any pair of elements has an infimum and a supremum    \cite{graetzer}. An admissible partition $\mP$ defines an equivalence relation over $\{1,\dots,2p\}$ by letting $i\sim_\mP j$  if $i,j\in P$ for some $P\in \mathcal{P}$. 
 The \emph{union} $\mP\vee \mP'$ is the partition such that $i\sim_{\mP\vee\mP} j$ if and only if  
 there exists a sequence $i=i_0,i_1,\dots,i_{k-1},i_k=j$ such that either 
 $i_{\ell-1} \sim_\mP i_{\ell}$ or  $i_{\ell-1} \sim_{\mP'} i_{\ell}$ for all $\ell=1,\dots,k$. Similarly, the \emph{intersection}  $\mP\wedge \mP'$ is the  partition such that $i\sim_{\mP\wedge \mP'} j$ if and only if  $i \sim_\mP j$ and $i \sim_{\mP'} j$. It is straightforward to show that $\mP\vee\mP'$ and $\mP\wedge \mP'$ are both  admissible.
 
Further, if $\mP\leq \mP'$, then there exists a (possibly non-unique) sequence of admissible partitions $\mP= \mP_0\leq\mP_1 \leq \dots  \leq \mP_k = \mP'$, such that at each step precisely two subsets of the partition are joined (cf. \cite[Lemma 1 and Lemma 403]{graetzer}). 
 This constitutes the proof of the following proposition.

\begin{proposition}\label{prop:onebyone}
\begin{enumerate}[(i)]
\item 
The set of equivalence classes of reaction graphs with the partial order $\preceq $   is a finite lattice with maximal element $\mG_{sp}$ and minimal element $\mG_c$. Further, the infimum and supremum of two classes $\mG_\mP,\mG_{\mP'}$ are respectively
\[\mG_\mP \wedge \mG_{\mP'} = \mG_{\mP\vee \mP'},\qquad \mG_\mP \vee \mG_{\mP'} = \mG_{\mP\wedge \mP'}.\]
\item Let $G,G'$ be reaction graphs associated with a reaction network.
Assume $k= m_{G} - m_{G'} > 0$. 
Then 
 $G'\preceq G$
if and only if there exists a sequence of $k-1$ reaction graphs $G'= G_0\preceq G_1 \preceq \dots \preceq G_{k-1}\preceq G_k = G$ such that
$$ m_{G_i} = m_{G_{i-1}} + 1,\qquad i=1,\dots,k.$$
\end{enumerate}
\end{proposition}

Note that given two weakly reversible reaction graphs $G,G'$, their infimum $G\wedge G'$ is also  weakly reversible (cf. Proposition~\ref{prop:incidence}(iv)), but this is not necessarily the case for the supremum $G\vee G'$ as the next example will show.

\begin{myexampleHorn}
Table~\ref{fig:partition} shows the  admissible partitions for the reaction graphs in Figure~\ref{fig:horn}. Since $\mP_5$ is a refinement of $\mP_4$, we have 
 $[G_4]\preceq [G_5]$. Further, $\mP_2 \wedge \mP_4=\mP_5$ and $\mP_2 \vee \mP_4=\mP_1$. 
It follows that  $[G_2]\wedge [G_4]=\mG_{\mP_2\vee \mP_4}$ is the equivalence class of the complex reaction graphs class with representative $G_1$, and similarly
$[G_2]\vee [G_4]=  \mG_{\mP_2\wedge \mP_4}=[G_5]$. While $G_2$ and $G_4$ are weakly reversible, $G_5$ is not.
Following Proposition~\ref{prop:onebyone}(ii), the inclusion $G_1\preceq G_3$  can be broken down into two sequences of inclusions $G_1\preceq G_2 \preceq G_3$ and $G_1\preceq G_2' \preceq G_3$, with $G_2'=G_\mP$ and 
$\mP=\big\{ \{1,8,9,11\}, \{2,3\}, \{4,12\}, \{6,7\},\{5,10\} \big\}$. 
\end{myexampleHorn}
 
 \begin{table}[t]
\begin{center}
\begin{tabular}{clc}
\toprule
$i$ & $\mP_i$ such that $G_i=G_{\mP_i}$  & $\delta_{G_i}$ \\ 
\toprule\toprule
$1$ &   \big\{ \{1,8,9,11\}, \{2,3\}, \{4,5,10,12\}, \{6,7\} \big\} & $2$\\ \midrule
$2$   & \big\{ \{1,11\}, \{2,3\}, \{4,5,10,12\}, \{6,7\},\{8,9\} \big\}  & $3$ \\ \midrule
$3$   &\big\{ \{1,11\}, \{2,3\}, \{4,12\}, \{6,7\},\{8,9\},\{5,10\} \big\}  &  $ 3$  \\ \midrule
$4$ &    \big\{ \{1,8\}, \{2,3\}, \{4,5,10,12\}, \{6,7\},\{9,11\} \big\}   & $3$ \\ \midrule
$5$ &    \big\{ \{1\}, \{2,3\}, \{4,5,10,12\}, \{6,7\},\{8\},\{9\},\{11\} \big\} & $5$ \\ \midrule
$6$ &  \big\{ \{1\}, \{2\}, \{3\}, \{4\}, \{5\}, \{6\}, \{7\}, \{8\}, \{9,11\}, \{10,12\} \big\}  & $4$\\ \midrule
$7$ &  \big\{ \{1\}, \{2\}, \{3\}, \{4\}, \{5\}, \{6\}, \{7\}, \{8\}, \{9\}, \{10\}, \{11\}, \{12\} \big\}  & $5$ \\ \bottomrule
\end{tabular}
\end{center}
\caption{Partitions defining the reaction graphs in Figure~\ref{fig:horn} and their deficiencies.}\label{fig:partition}
\end{table}

We conclude the section  with an alternative description of the partial order. 
 
\begin{lemma}\label{lem:morphism-inclusion}
$G'\preceq G$ if and only if there exists a morphism of reaction graphs  $\varphi\colon V_G\to V_{G'}$.
\end{lemma}
\begin{proof}
Let $\mP=\{P_1,\dots,P_{m_G}\}$ and $\mP'=\{P'_1,\dots,P'_{m_{G'}}\}$ be the  
admissible partitions   such that $G=G_\mP$ and $G'=G_{\mP'}$.
Since $G'\preceq G$, we have $\mP\leq \mP'$. 
For each $i\in V_G$, let $\varphi(i)\in V_{G}'$ be the index such that $P_i\subseteq P_{\varphi(i)}$.
Since the partitions are admissible, then
$$Y_i = (Y_{\mP})_i =Y_{sp,k} =(Y_{\mP'})_{\varphi(i)}= Y'_{\varphi(i)},\quad \textrm{for all }k\in P_i\subseteq P_{\varphi(i)}'. $$
Further, if $i\rightarrow j\in E_G = E_\mP$, then by definition there exist $i'\in P_i$ and $j'\in P_j$ such that $i'\rightarrow j'$ is an edge of $G_{sp}$. Since also $i'\in P'_{\varphi(i)}$ and $j'\in P'_{\varphi(j)}$, it follows that $\varphi(i)\rightarrow \varphi(j) \in E_{\mP'} = E_{G'}$. By Definition~\ref{def:morphismgraph}, $\varphi$ is a morphism of reaction graphs.

For  the reverse implication, let $i\in V_G$ and $k\in P_i$. Since $\varphi$ is a morphism and $G=G_\mP$ it holds $(Y_{\mP'})_{\varphi(i)}=(Y_{\mP})_i =Y_{sp,k}$. 
 Hence $k\in P'_{\varphi(i)}$. This shows $P_i\subseteq P'_{\varphi(i)}$ and thus $\mP\leq \mP'$ as desired.
\end{proof}

The morphism $\varphi$ associated with an inclusion $G'\preceq G$ identifies the nodes of $G$ that are joined to form $G'$, or, in terms of  partitions, the subsets of the partition defining $G$ that are joined to form the partition defining $G'$.
By Lemma~\ref{lem:morphism-inclusion}, two reaction graphs $G$ and $G'$ are isomorphic  if and only if $G\preceq G'$ and $G'\preceq G$, which is consistent with Lemma \ref{lem:equiv}.

\begin{myexampleHorn}\label{ex:horn-inclusions}
For the  inclusion  $G_4 \preceq G_5$ in Figure~\ref{fig:horn}, 
the map $\varphi\colon V_{G_5}\rightarrow V_{G_4}$ is
$$ \varphi(1) =\varphi(5)=1,\quad  \varphi(2)=2,   \quad \varphi(3)=3,\quad \varphi(4)=4, \quad \varphi(6)=\varphi(7)=5. $$
\end{myexampleHorn}

\subsection{Incidence matrix and weak reversibility}
Given a reaction graph $G$ with labeling $Y$, we consider the matrix $Y\in \R^{n\times m_G}$, called the \emph{labeling matrix}, with $j$-th column being the vector $Y_j$
(the same notation is used as for the labeling function, since both objects convey the same information).

We define a numbering on the edge set $E_G$ by means of the bijection $R$ and the numbering of $\mR$.
The \emph{incidence matrix} $C_G$ of the reaction graph $G$ is the $m_G\times p$ matrix with nonzero entries defined by 
$(C_G)_{k_1j} =-1$,    $(C_G)_{k_2j} =1$,    if the $j$-th edge is $k_1\rightarrow k_2$. Each column of $C_G$ has only two non-zero entries and the column sums  are zero. The rank of the incidence matrix  is $m_G - \ell_G$, where $\ell_G$ is the number of connected components of $G$ \cite[Prop. 4.3]{Biggs}.

\begin{proposition}\label{prop:incidence} 
Let $G$ be a reaction graph  associated with a reaction network $\mN$. Then, the following holds:
\begin{enumerate}[(i)]
\item 
$N=Y C_G.$
\item The kernel of $C_G$ contains a positive vector if and only if $G$ is weakly reversible.
\end{enumerate}
Assume $G,G'$ are two reaction graphs associated with  $\mN$ such that $G'\preceq G$. Then, we have
\begin{enumerate}
\item[(iii)]  There exists an $(m_{G'}\times m_{G})$-matrix $B$ such that $C_{G'}=BC_{G}$.
\item[(iv)]  If $G$ is weakly reversible, then so is $G'$.
\end{enumerate}
\end{proposition}
\begin{proof}
(i) The $j$-th column of $N$ is $y'-y$ if $r_j\colon  y\rightarrow y'$. 
If $k_2\rightarrow k_1$ is the $j$-th edge of $G$, then the $j$-th column of $YC_G$ is the  vector $Y_{k_2}-Y_{k_1}$. 
Since $r_j=R_{k_2\rightarrow k_1}$, we have $Y_{k_2}= y$ and $Y_{k_1}=y'$, proving (i).

 (ii)  This is well known. For instance, use that the elementary cycles are the minimal generators of the polyhedral cone $\ker(C_G)\cap \R^{p}_{\geq 0}$ \cite[Prop. 4]{stadler} and that a digraph is weakly reversible if and only if each edge is part of an elementary cycle. See also \cite[Remark 6.1.1]{Feinbergss} or \cite[Lemma 3.3]{muller}. 
 
(iii) 
Let $V_{G} \xrightarrow{\varphi} V_{G'} $ be the surjective map from Lemma~\ref{lem:surjective}.
Let $B\in \R^{m_{G'}\times m_{G}}$ such that the $(\varphi(i),i)$-th entry is $1$, for $i=1,\dots,m_{G}$,  and the rest of the entries are zero.
If the $j$-th edge of $G$ is $k_1\rightarrow k_2$, then the nonzero entries of the $j$-th column of $BC_{G}$ are $-1$ at row $\varphi(k_1)$ and
$1$ at row  $\varphi(k_2)$. Thus $BC_G=C_{G'}$.

(iv) Follows from  (ii) and (iii).
\end{proof}

\subsection{Deficiency}

The deficiency of a reaction network is an important characteristic in the study of steady states and their properties \cite{feinberg1972,feinberg-def0, Craciun-Sturmfels, shinar-science, boros:deficiency}.
Here, we extend the classical definition of deficiency to an arbitrary reaction graph. 

\begin{definition}\label{def:deficiency}
Let $G$ be a reaction graph associated with a reaction network $\mN$. The \emph{deficiency} of 
 $G$ is  the number
$$\delta_{G}= m_G - \ell_G-s,$$ where $m_G$ is the number of nodes of $G$,
$\ell_G$ is the number of connected components of $G$ and $s$ the rank of the stoichiometric subspace of $\mN$. 
\end{definition}

The deficiency of a reaction graph depends only on the equivalence class of the reaction graph. Hence it makes sense to talk about the deficiency of an equivalence class $\mG$, which we denote by $\delta_\mG$. We have $\delta_{[G]} = \delta_G$.
In particular, 
we refer to the \emph{deficiency of $\mN$} as the the deficiency of the equivalence class  $\mG_c$ of the complex reaction graphs (in accordance with \cite{feinberg1972}). 

The deficiencies of the reaction graphs in Figure~\ref{fig:horn} are given in Table~\ref{fig:partition}, using that $s=1$.

\begin{example}\label{ex:defsplitdet}
A split reaction graph has $2p$ nodes and $p$ connected components. 
Thus the deficiency of the split equivalence class $\mG_{sp}$ is 
$\delta_{\mG_{sp}} = 2p-  p-s = p-s.$

Given a reversible network with $p=2q$ reactions, then any detailed reaction graph has $p$ nodes and $q$  connected components. Thus the deficiency of the detailed equivalence class is 
$ \delta = p -q-s = q-s.$
\end{example}

The deficiency of a  reaction graph is  a non-negative number, as we show below.
For a given order $G_1,\dots,G_{\ell_G}$ of the connected components of $G$, define the linear map $\Psi_G$ by
\begin{equation}\label{eq:chiG}
 \Psi_G \colon \R^{m_G} \rightarrow \R^{n+\ell_G} \qquad \textrm{with} \qquad  \Psi_G (x)= \left(Yx, \sum_{i\in G_1} x_i,\dots,\sum_{i\in G_{\ell_G}} x_i\right), 
 \end{equation}
where $Yx$ defines the first $n$ coordinates of $\Psi_G (x)$. The following proposition is well known in the context of complex balanced steady states \cite{hornjackson,Craciun-Sturmfels}. We will use the result in Section~\ref{sec:node-balanced}.

\begin{proposition}\label{lem:defkernel}
For a reaction graph $G$ with labeling function $Y$, the following equalities hold:
$$ \ker \Psi_G= \ker Y \cap \im C_G,\qquad  \delta_G = \dim (\ker Y \cap \im C_G)=\dim \ker \Psi_G.$$
\end{proposition}
\begin{proof}
To prove the first equality, let $e_i$ be the $i$-th unit vector of $\R^{m_G}$. For all 
$j\in\{1,\ldots,\ell_G\}$,
we have  $\omega_j:=\sum_{i\in G_j} e_i \in (\im C_G)^\perp$. This gives $\ell_G$ linearly independent vectors in $(\im C_G)^\perp$. Using that the rank of $C_G$ is $m_G-\ell_G$, we conclude that $\omega_1,\dots,\omega_{\ell_G}$ form a basis of $(\im C_G)^\perp$. It follows that $x\in \im C_G$ if and only if $\omega_j\cdot x=0$ for all $j\in\{1,\ldots,\ell_G\}$. Further, $\omega_j\cdot x=\sum_{i\in G_j} x_i$.

Using this, we have $x\in  \ker Y \cap \im C_G$ if and only  if $Yx=0$ and $\omega_j\cdot x=0$, which is equivalent to $\Psi_G(x)=0$. This concludes the proof of the first equality.

Consider now the linear map $\psi\colon \ker N \rightarrow \R^{m_G}$ defined by $w\mapsto C_Gw$ on $\ker N \subseteq \R^{p}$.
Since $\ker C_G \subseteq \ker N$ by Proposition \ref{prop:incidence}(i), we have 
$\ker \psi =  \ker C_G$.
On the other hand, 
\begin{align*}
 \im \psi  & = \{  C_G w \st w\in\R^p,  Nw=0\} =  \{  C_G w \st w\in\R^p,  YC_G w=0\}  =
\ker Y \cap \im C_G.
\end{align*}
By the first isomorphism theorem, $\ker N /  \ker C_G \cong \ker Y \cap \im C_G$. Using that 
$\ker C_G$ has dimension $p-m_G+\ell_G$, we have 
$$ \dim (\ker Y \cap \im C_G) = \dim \ker N - \dim \ker C_G = p - s - (p -m_G+\ell_G) = m_G  - \ell_G - s= \delta_G$$
which concludes the proof  of the second equality.
\end{proof}

The deficiency behaves surprisingly nice in connection with inclusion of reaction graphs.
In particular, it is possible to iteratively compute the deficiency of any reaction graph from the split reaction graph. This will be further exploited in the coming section in relation to node balanced steady states.

\begin{proposition}\label{prop:deficiency}
Assume $G'\preceq G$ for two reaction graphs associated with a reaction network $\mN$ and let $\varphi\colon V_{G} \rightarrow V_{G'}$ be the morphism defining the inclusion, cf. Lemma~\ref{lem:morphism-inclusion}.
\begin{enumerate}[(i)]
\item 
Assume that $m_{G} = m_{G'}+1$
 and let $i_1,i_2$ be the only two nodes of $G$ such that $\varphi(i_1)=\varphi(i_2)$. 
 If $i_1,i_2$ belong to the same connected component of $G$, then $\delta_{G'}=\delta_{G}-1$ and there is an undirected cycle through $\varphi(i_1)=\varphi(i_2)$ in $G'$, which is not in $G$.
Otherwise,  $\delta_{G'}=\delta_G.$
\item $\delta_{G'} \leq \delta_{G}$.
\item If $\Delta= \delta_{G} - \delta_{G'} > 0$, then there exists a sequence of $\Delta+1$ reaction graphs $G'= G_0\preceq G_1 \preceq \dots \preceq G_{\Delta-1}\preceq G_\Delta = G$ such that
$$ \delta_{G_i} = \delta_{G_{i-1}} + 1,\qquad i=1,\dots,\Delta.$$
\end{enumerate}
\end{proposition}
\begin{proof}
(i) 
Using that $m_{G} = m_{G'}+1$, we have that
$$\delta_{G} - \delta_{G'} =m_{G'}+1 - \ell_{G} - s -(m_{G'} - \ell_{G'} - s ) =  \ell_{G'} - \ell_{G} + 1. $$
Let $G_1,G_2$ be the connected components of $G$ containing $i_1,i_2$ respectively. 
Outside $G_1$ and $G_2$, the morphism $\varphi$ is  a permutation of the nodes, while  $G_1$ and $G_2$ are mapped to the connected component of $G'$ containing $\varphi(i_1)=\varphi(i_2)$. Thus, if $G_1=G_2$, $\ell_G = \ell_{G'}$, while if $G_1\neq G_2$, then $\ell_{G'} = \ell_{G} -1$. This gives the stated relation between $\delta_G$ and $\delta_{G'}$.

Furthermore, consider an undirected path without repeated nodes between $i_1$ and $i_2$ in $G$. The image by $\varphi$ of this path is a path of $G'$, with end points $\varphi(i_1)=\varphi(i_2)$.  Thus there is an undirected cycle in $G'$. There are no repeated nodes since $\varphi$ is one-to-one on $ V_{G}\setminus \{i_1,i_2\}$.

(ii) and (iii) follow from (i) and Proposition~\ref{prop:onebyone}.
\end{proof}

The above results have some interesting consequences. 
Assume $G'\preceq G$. If $G,G'$ have exactly the same cycles (given the correspondence of  edges between the two graphs), then their deficiency is the same by Proposition \ref{prop:deficiency}(i). 
Further, the deficiency of $G'$ can be computed iteratively from that of $G$ by joining pairs of nodes, one at a time, and checking whether two connected components merge or not.

\begin{myexampleHorn}
Consider the inclusion $G_1\preceq G_2\preceq G_3$ and the deficiencies given  in Table~\ref{fig:partition}.
The inclusion $G_2\preceq G_3$ is obtained by  joining the nodes  $3,6$, which changes the number of connected components. Thus, $\delta_{G_2}=\delta_{G_3}=3$. 
For the inclusion $G_1\preceq G_2$ one joins the nodes $1,5$, which does not alter the number of connected components and a new cycle is created in $G_1$, namely $1\rev 3$. The deficiency is reduced by one,  hence
$\delta_{G_1} = \delta_{G_2}-1 = 2$.
\end{myexampleHorn}

\section{Node balanced steady states}\label{sec:node-balanced}
\label{sec:node}

In this section we introduce  \emph{node balanced steady states} for a given reaction graph, and show that
complex and detailed balanced steady states, two well-studied classes of steady states, are examples of node balanced steady states.
We use the partial order on the equivalence classes of reaction graphs  to deduce further properties of node balanced steady states.

\subsection{Definition and first properties}

\begin{definition}\label{def:node-balanced}
Let $\mN$  be a  reaction network, $G=(V_G,E_G,Y)$ an associated reaction graph and $v$ a kinetics.
A \emph{node balanced steady state} of $\mN$ with respect to $G$ is a solution to the equation
$$C_G\, v(x)=0,\qquad x\in \R^n_{\geq 0}.$$
Equivalently,  $x^*\in \R^n_{\geq 0}$ is 
a node balanced steady state if for all nodes $i$ of $C_G$ it holds
\begin{equation}\label{eq:node-condition}
 \sum_{j\in V_G \st i\rightarrow j \in E_G} v_{Y_i\rightarrow Y_j} (x^*) =  \sum_{j\in V_G \st j\rightarrow i \in E_G} v_{Y_j\rightarrow Y_i} (x^*),   
 \end{equation}
where $v_{Y_i\rightarrow Y_j}$ is the rate function of the reaction $Y_i\rightarrow Y_j$.
\end{definition}

By Proposition~\ref{prop:incidence}(i),  any node balanced steady state is also a steady state of the network. 
 Further, a direct consequence of Proposition~\ref{prop:incidence} 
 is the following result.

\begin{proposition}\label{prop:nodeinclusion}
Let $\mN$  be a  reaction network with a kinetics $v$ and let $G'\preceq G$ be two associated reaction graphs.
\begin{enumerate}[(i)]
\item If there exists a positive node balanced steady state with respect to $G$, then $G$ is weakly reversible.
\item Any node balanced steady state of $\mN$ with respect to $G$ is also a node balanced steady state with respect to $G'$. 
\end{enumerate}
\end{proposition}

By
Proposition~\ref{prop:nodeinclusion}(ii),  a node balanced steady state with respect to $G$ is a node balanced steady state with respect to any reaction graph equivalent to $G$. 
Thus,  we may equivalently refer to a node balance steady state with respect to an equivalence class of reaction graphs  or to a particular reaction graph in that class. 
Further, a \emph{complex balanced steady state } is a node balanced steady state with respect to the  equivalence class of complex reaction graphs, and similarly, a \emph{detailed balanced steady state} is a node balanced steady state with respect to the  equivalence class of detailed reaction graphs.

By Proposition~\ref{prop:nodeinclusion}(ii)  and equation~\eqref{eq:complex-small}, any node balanced steady state of $\mN$  is also a complex balanced steady state and we recover the well-known fact that detailed balanced steady states are complex balanced steady states.
In fact, more is true. We will show in Corollary~\ref{cor:node} that  if $G'\preceq G$ and $\delta_{G'}=\delta_G$, then the reverse implication of Proposition~\ref{prop:nodeinclusion}(ii)   holds for mass-action kinetics. That is, both reaction graphs give rise to the same set of node balanced steady states, and there exist node balanced steady states for the same set of reaction  rate constants. 

Since a split reaction graph is never weakly reversible, there cannot be positive node balanced steady states
with respect to this graph. Similarly, if a reaction network contains irreversible  reactions, then there are no positive detailed balanced steady states with respect to the detailed reaction graph.

\begin{myexampleHorn}\label{ex:hornexnode}
Consider the reaction graphs $G_2,G_3$ in Figure~\ref{fig:horn} and mass-action kinetics.  Focusing on the node $3$ with label $3X_2$,  
a positive node balanced steady state with respect to $G_2$
fulfils 
\begin{equation}\label{eq:exG3a}  
\k_{2} x_1x_2^2 +  \k_{5} x_1^3  =  \k_{6} x_2^3 +  \k_{3} x_2^3.
\end{equation}
For nodes $3$ and $6$, both with label $3X_2$, a positive node balanced steady state with respect to $G_3$
fulfills 
\begin{equation}\label{eq:exG3} 
\k_{2} x_1x_2^2 =  \k_{6} x_2^3,  \qquad \textrm{and}\qquad   \k_{5} x_1^3   =  \k_{3} x_2^3.  
\end{equation}
If these equations hold, then so does \eqref{eq:exG3a}. The equations for the other nodes are the same for both $G_2$ and $G_3$. Therefore a node balanced steady state with respect to $G_3$ is also a node balanced steady state with respect to $G_2$. We knew  this already from Proposition~\ref{prop:nodeinclusion}  since $G_2\preceq G_3$. 
The reverse statement is also true in this case. It follows from the fact that the deficiencies of the two reaction graphs agree. The proof will be given below.
\end{myexampleHorn}

\subsection{Main results}

In this section we  only consider reaction networks  with mass-action kinetics and weakly reversible reaction graphs. 
Given a network $\mN$, an associated reaction graph $G$ and a vector  $\k\in \R^p_{>0}$, we say that $(\mN,G,\k)$ is  \emph{node balanced} if there exists a \emph{positive} node balanced steady state of $\mN$ with respect to $G$ for  mass-action kinetics with reaction  rate constants $\k$.
The proofs of the three theorems below are given in Section~\ref{sec:proof}.

The first theorem is akin to Theorem 6A of \cite{hornjackson}.  

\begin{theorem}[{\bf Uniqueness and asymptotic stability}]\label{thm:complex}
Assume $\mN$ is a reaction network with mass-action kinetics with  reaction   rate constants $\k\in \R^p_{>0}$, and let $G$ be a  weakly reversible reaction graph.
 If $(\mN,G,\k)$ is node balanced,
  then any positive steady state of $\mN$ is node balanced with respect to $G$. Furthermore, there is one such steady state in each stoichiometric compatibility class and it is locally asymptotically stable relatively to the class.
\end{theorem}

The next main result tells us that there are $\delta_G$ equations in $\k\in \R^p_{>0}$ that determine whether $(\mN,G,\k)$ is node balanced or not.
The link between deficiency and the existence of complex balanced steady states was first explored by Horn and Feinberg \cite{horn,feinberg-balance} (see also \cite{feinberg-def0}, and \cite{Feinberg:1989vg,Dickenstein:2011p1112} concerning detailed balancing). We follow the approach of \cite{Craciun-Sturmfels} for complex balanced steady states,  which applies to general reaction graphs.

For a node $i\in V_G$, let $\Theta_{G,i}$ be the set of spanning trees of the connected component $i$  belongs to, rooted at $i$ (that is,  $i$ is the only node with no outgoing edge).
Given such a tree $\tau$, let $\k^\tau$ be the product of the reaction  rate constants corresponding to the edges of  $\tau$.
Define
\begin{equation}\label{eq:KG}
 K_{G,i}  = \sum_{\tau \in \Theta_{G,i} } \k^\tau
\qquad\textrm{and}\qquad K_G = (K_{G,1},\dots,K_{G,m_G}).
 \end{equation}

The following theorem is a consequence of \cite[Theorem 9]{Craciun-Sturmfels}, adapted to our setting. 
Recall that the kernel of the map $\Psi_G$ given in \eqref{eq:chiG}  has dimension $\delta_G$ (Proposition~\ref{lem:defkernel}).

\begin{theorem}[{\bf Existence of node balanced steady states}]\label{prop:k}
Let $\mN$ be a reaction network, $\k\in \R^p_{>0}$, and $G$ an  associated weakly reversible reaction graph. Let $u_1,\dots,u_{\delta_{G}}\in \Z^{m_G}$ be a basis of $\ker \Psi_G$.  Then the following holds.
\begin{enumerate}[(i)]
\item $(\mN,G,\k)$ is node balanced if and only if
$$ (K_G)^{u_i} =\prod_{j=1}^{m_G} K_{G,j}^{u_{ij}} = 1, \qquad \textrm{for all }\quad i=1,\dots,\delta_G.$$
\item 
$x\in \R^n_{>0}$ is a node balanced steady state of $\mN$ with respect to $G$ if and only if 
$$ K_{G,i} x^{Y_i} - K_{G,j} x^{Y_j}=0 \qquad \textrm{for all }\quad (i,j)\in E_G.$$
\end{enumerate}
\end{theorem}

By Theorem~\ref{prop:k}(i), there are $\delta_G$  relations in the reaction  rate constants for which $(\mN,G,\k)$ is node balanced. In fact, more is true. As proven in \cite{Craciun-Sturmfels} for complex balanced steady states, these $\delta_G$ equations imply 
  that the set of reaction  rate constants $\k\in \R^p_{>0}$ for which a positive node balanced steady state exists forms a positive variety of codimension $\delta_G$ in $\R^p_{>0}$.  The conditions given in (i) for complex balanced or detailed balanced steady states are equivalent to the conditions given in Proposition 1 and 4 of \cite{Dickenstein:2011p1112}.

The next result  tells us how to obtain the relations in Theorem~\ref{prop:k}(i) from the relations  for complex balanced steady states.

\begin{theorem}[{\bf Node balancing and inclusion of reaction graphs}]\label{thm:inclusion}
Let $G',G$ be two weakly reversible reaction graphs such that $G'\preceq G$ and $m_{G} = m_{G'} +1$. Further, consider the associated morphism $\varphi\colon G\rightarrow G'$ in Lemma~\ref{lem:morphism-inclusion}, and let   $i_1,i_2$ be the only pair of nodes of $G$ such that $\varphi(i_1)=\varphi(i_2)=k$ with $k\in V_{G'}$.  Let $\k\in \R^p_{>0}$.

\begin{enumerate}[(i)]
\item If $i_1,i_2$ do not belong to the same connected component of $G$, then $(\mN,G,\k)$ is node balanced if and only if $(\mN,G',\k)$ is.
\item If $i_1,i_2$ belong to the same connected component of $G$, then $(\mN,G,\k)$ is node balanced if and only if $(\mN,G',\k)$ is node balanced and further 
$K_{G,i_1}=K_{G,i_2}$.
\end{enumerate}
\end{theorem}

We obtain the following corollary.

\begin{corollary}\label{cor:node}
Let $G'\preceq G$ be two weakly reversible reaction graphs associated with a reaction network.
Assume mass-action kinetics with $\k\in \R^p_{>0}$.

\begin{enumerate}[(i)]
\item 
If $\delta_G = \delta_{G'}$, then
$x^*$ is a positive node balanced steady state with respect to $G'$ if and only if it  is a positive node balanced steady state with respect to $G$.

\item Any set of $\delta_{G'}$ equations in $\k$  for the existence of positive node balanced steady states with respect to $G'$ (as in Theorem~\ref{prop:k}) can be extended to a set of equations in $\k$ for the existence of positive node balanced steady states with respect to $G$ by adding $\delta_G-\delta_{G'}$ equations.
\end{enumerate}
\end{corollary}
\begin{proof}
(i) The reverse implication is Proposition~\ref{prop:nodeinclusion}(ii). If $x^*$ is a positive node balanced steady state with respect to $G'$, 
then by Theorem~\ref{thm:inclusion}(i) and Proposition~\ref{prop:deficiency}(i), $(\mN,G,\k)$ is also node balanced. By Theorem~\ref{thm:complex}, this means that all positive steady states of $\mN$ with reaction  rate constants $\k$ are node balanced with respect to $G$, thus in particular $x^*$ is.

(ii) It follows from  Theorem~\ref{thm:inclusion}(ii) and Proposition~\ref{prop:deficiency}(iii).
\end{proof}

As a consequence of Corollary~\ref{cor:node} and Proposition~\ref{prop:deficiency}(i),  we recover the following well-known result relating complex and detailed balanced steady states \cite{Feinberg:1989vg,Dickenstein:2011p1112}. Assume the network is reversible. If a complex reaction graph has no simple cycles (that is, with no repeated nodes) other than those given by pairs of reversible reactions, then  a steady state is complex balanced if and only if it is detailed balanced.

By letting $G'$ be a complex reaction graph in Corollary~\ref{cor:node}(ii),  we conclude that the equations for the existence of positive complex balanced steady states can be extended to equations for the existence of positive node balanced steady states  for any weakly reversible reaction graph. The extra conditions, obtained by applying  Theorem~\ref{thm:inclusion}(ii) iteratively in conjunction with Proposition~\ref{prop:deficiency}, generalize  the so-called cycle conditions \cite{Feinberg:1989vg}, or formal balancing conditions \cite{Dickenstein:2011p1112}, that relate conditions for the existence of complex and detailed balancing steady states.

If $\delta_G=\delta_{G'}$ but neither $G\preceq G'$ nor $G'\preceq G$, all possibilities might occur. The sets of reaction  rate constants $\k\in\R^p_{>0}$ for which $(\mN,G,\k)$ and $(\mN,G',\k)$ are node balanced can be disjoint, have a proper intersection or coincide. Instances of the first two cases are given in Example~\ref{ex:equations_differ} and Example~\ref{ex:no_intersection} below. For what concerns the latter, we have the following corollary, which is a consequence of Theorem~\ref{thm:inclusion}.

\begin{corollary}\label{cor:coincide} 
Let $G,G'$ be two weakly reversible reaction graphs with the same deficiency $\delta_G=\delta_{G'}$.  If either of the following conditions are fulfilled,  
\begin{enumerate}[(i)]
\item  $\delta_G = \delta_{[G]\wedge [G']}$.
\item  $[G]\vee [G']$ is weakly reversible and $\delta_G = \delta_{[G]\vee [G']}$,
\end{enumerate}
then $(\mN,G,\k)$ is node balanced if and only if $(\mN,G',\k)$ is.
\end{corollary}

\begin{example}\label{ex:detailed}
Consider a reaction network with reactions $X_1\rev X_2$, $X_2\rev X_3$ and $X_3\rev  X_4$, and the following associated equivalence classes of reaction graphs (shown without the numbering of the nodes).
 \begin{center}
\begin{tabular}{rc}
$\mG_1$:&
$X_1 \rev X_2 \rev X_3 \rev X_4$
\\
 \midrule
 $\mG_2$:&
$X_1\rev X_2\qquad$ 
$X_2\rev X_3\rev X_4$
\\
 \midrule
 $\mG_3$:&
 
$X_1\rev X_2\rev X_3\qquad$
$X_3\rev X_4$
\\
 \midrule
 $\mG_4$:&
$X_1 \rev X_2\qquad$
$X_2\rev X_3\qquad$
$X_3\rev X_4$
\\
\end{tabular}
\end{center}
Neither $\mG_2\preceq \mG_3$ nor $\mG_3\preceq \mG_2$. Moreover, $\mG_2\wedge \mG_3=\mG_1$, $\mG_2\vee \mG_3=\mG_4$, and the reaction graphs in these classes have all the same deficiency. Hence, by Corollary~\ref{cor:coincide}, the sets of reaction   rate constants for which $\mN$ admits positive node balanced steady states with respect to any of the above equivalence classes coincide.
\end{example}

\paragraph{Finding $K_G$.}
We conclude this section by expanding on how to find the relations in Theorem~\ref{prop:k}(i) in practice and with further examples.  Let $G_1,\dots,G_{\ell_G}$  be the  connected components of a weakly reversible reaction graph $G$.
To determine the kernel of $\Psi_G$, we consider the associated $(n+\ell_G)\times m_G$  Cayley matrix $A_G$ \cite{Craciun-Sturmfels}. It is a block matrix with upper block $Y$ ($n\times m_G$) and lower block ($\ell_G\times m_G$), where the $i$-th row has  $1$ in entry $j$ if node $j$ belongs to component $G_i$, and otherwise is zero.
Since $A_G$ is  an integer matrix, then  $\ker(\Psi_G)$  has a basis with integer entries as well.

The form of $K_{G,i}$ is a consequence of the Matrix-Tree Theorem  \cite{Tutte-matrixtree,StanleyCombinatorics}.
Specifically, let $L_\k  \in \R^{m_G \times m_G}$ be the Laplacian of $G$ with  off-diagonal $(k,j)$-th entry equal to $\k_\ell$, if $j\rightarrow k\in E_G$ and the edge corresponds to the reaction $r_\ell$, and zero otherwise. Let 
$L_\k'$ be the submatrix of $L_\k$ obtained by selecting the rows and columns with indices in the node set of the connected component $G'$ of $G$ that contains node $i$. 
Let $(L_\k')_{(i)}$ denote the minor corresponding to the submatrix  obtained from $L_\k'$ by removing the $i$-th   column and the last row.
Then 
\begin{equation}\label{eq:Kminor}
K_{G,i} = (-1)^{i}   (L_\k')_{(i)}. 
\end{equation}
In fact $ (-1)^{i}   (L_\k')_{(i)}$ also agrees  with $(-1)^{m'+1+i+j}$ times the minor obtained from $L_\k'$ by removing the $i$-th   column and the $j$-th row, where $m'$ is the number of nodes of $G'$.

In practice, this is often the way to find $K_{G}$, rather than finding the trees. However, the description of $K_G$ in terms of trees is a useful  interpretation of  $K_G$.

\begin{myexampleHorn}\label{ex:G4:K}
The reaction graph $G_4$ in  Figure~\ref{fig:horn} has one connected component. The map $\Psi_G$ has matrix 
$$A_G = \begin{pmatrix}
3 & 1 & 0 & 2 & 3 \\
0 & 2 & 3 & 1 & 0 \\
1 & 1 & 1 & 1 & 1
\end{pmatrix}\,\in\, \Z^{3\times 5}.$$
The first two rows of $A_G$ are the column vectors encoded by the labels (complexes) of the nodes of the graph. 
Since $\delta_G=3$, we find three linearly independent vectors in the kernel  with integer entries,
$u_1 = (-1,0,0,0,1), u_2 = (-1,-1,0,2,0), u_3 = (1,-3,2,0,0).$
Furthermore, we have
$K_{G} =\big( \k_2\k_3\k_4\k_5,  \, \k_1\k_3\k_4\k_5, \,  \k_1\k_2\k_4\k_5,  \, \k_1\k_2\k_3\k_5, \,  \k_1\k_2\k_4\k_6\big).$
By Theorem~\ref{prop:k}, $(\mN,G,\k)$ is node balanced if and only if 
$$ K_{G,1}^{-1}K_{G,5} = 1, \qquad K_{G,1}^{-1}K_{G,2}^{-1}K_{G,4}^2=1,\qquad K_{G,1}K_{G,2}^{-3}K_{G,3}^2=1.$$
This reduces to  three algebraic relations
\begin{equation}\label{eq:myK}
 \k_1\k_6 = \k_3\k_5,\qquad \k_1\k_2=\k_4^2  ,\qquad \k_1\k_3^2 = \k_2^3.  
 \end{equation}
We find analogously relations for positive node balanced steady states with respect to the complex balanced reaction graph $G_1$,
\begin{equation}\label{eq:myK2}
(\k_1+\k_5)^{2}\k_2^3 =(\k_3+\k_6)^2 \k_1^3, \qquad  \k_2\k_3^2 ( \k_1+\k_5)^2 = ( \k_3+\k_6) ^{2}\k_4^2\k_1
 \end{equation}
 The relation $K_{G_4,1}=K_{G_4,5}$ is $\k_3\k_5=\k_1\k_6$. By isolating $\k_6$ from this equation and inserting it into  \eqref{eq:myK2}, we obtain after simplification the set of equations in \eqref{eq:myK}, in accordance with Theorem~\ref{thm:inclusion}(ii).
\end{myexampleHorn}

\begin{myexampleHorn}\label{ex:equations_differ}
Since $G_2\preceq G_3$ and the reaction graphs have the same deficiency (Table~\ref{fig:partition}), it follows from  Corollary~\ref{cor:node}, that a steady state is node balanced with respect to $G_2$ if and only if it is with respect to $G_3$. This implies that equation  \eqref{eq:exG3} and  \eqref{eq:exG3a}   have the same solutions.
Similarly, a positive steady state is node balanced with respect to $G_5$ if and only if it is with respect to $G_4$. In this example,  $G_2\preceq G_3$ and $G_4\preceq G_5$.

The reaction graphs $G_2$ and $G_4$ which are not related by inclusion, have  deficiency $3$. However, the values of $\k$ for which  positive node balanced steady states exist differ. As in Example~\ref{ex:G4:K}, we find the relations for $G_2$ (identical to those for $G_3$)
\begin{equation}\label{eq:myK3} 
\k_1\k_3=\k_5\k_6,\qquad \k_1\k_2\k_3^2 = \k_4^2\k_6^2,\qquad \k_2^3=\k_1\k_6^2.
 \end{equation}
These equations and the equations for $G_4$ in equation~\eqref{eq:myK} in Example~\ref{ex:G4:K} define  different sets with non-empty intersection. For example, $\k=(1,\dots,1)$ fulfils both \eqref{eq:myK} and \eqref{eq:myK3}, but $\k=(1,1,1,1,2,2)$ fulfils only \eqref{eq:myK}. Note however that the matrix $A_G$ is the same for both reaction graphs.
\end{myexampleHorn}

\begin{example}\label{ex:no_intersection} 
 Consider a reaction network with reactions $X_1\rev X_2$, $X_2\rev X_3$, $X_1\rev X_3$ and $X_2\to X_4\to X_3$,  and the associated reaction graphs, depicted in Figure~\ref{fig:intersection}.

\begin{figure}[!t]
\begin{center}
\includegraphics[scale=1]{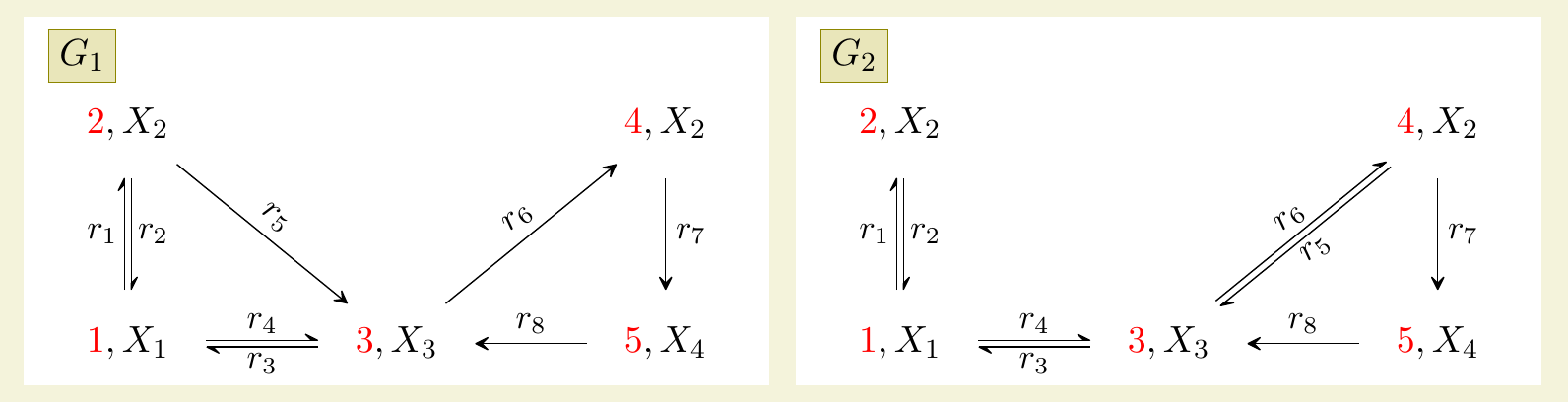}
\end{center}
\caption{Reaction graphs for the network in Example~\ref{ex:no_intersection}. } \label{fig:intersection}
\end{figure}

We have $\delta_{G_1}=\delta_{G_2}=1$, but neither  $G_1\preceq G_2$ nor $G_2\preceq G_1$. If $x^*$ were a positive node balanced steady state with respect to both $G_1$ and $G_2$,  it would follow from \eqref{eq:node-condition}, applied to node $2$, that $v_{X_2\to X_3}(x^*)=0$, which contradicts that $v(x)\in\R^p_{>0}$. 
Hence $(\mN, G_1, \kappa)$ and $(\mN, G_2, \kappa)$ cannot be node balanced for the same choice of reaction  rate constants.

 A different route to reach the same conclusion consists in finding the set of reaction rate constants $\k$ for which $(\mN, G_1, \kappa)$ and $(\mN, G_2, \kappa)$ are node balanced. The complex balanced reaction graph associated with the network has deficiency zero and thus $\mN$ is complex balanced for all $\k$. A complex balanced reaction graph can be obtained from either reaction graph,  $G_1$ or $G_2$, by joining the nodes $2$ and $4$. In view of Theorem~\ref{thm:inclusion}(ii), it is enough to find the label of the trees rooted at $2$ and $4$ for each reaction graph. We obtain that  $(\mN, G_1, \kappa)$  is node balanced if and only if 
\begin{align}
K_{G_1,2} = K_{G_1,4}  \qquad & \Longleftrightarrow \qquad   \k_1\k_3\k_7\k_8= \k_1\k_5\k_6\k_8 + \k_4\k_5\k_6\k_8+\k_2\k_4\k_6\k_8  \nonumber \\
& \Longleftrightarrow \qquad   \k_1\k_3\k_7= \k_1\k_5\k_6 + \k_4\k_5\k_6+\k_2\k_4\k_6.\label{eq:G1node}
\end{align}
Similarly,  $(\mN, G_2, \kappa)$  is node balanced if and only if 
\begin{align}
K_{G_2,2} = K_{G_2,4}  \qquad & \Longleftrightarrow \qquad     \k_1\k_3\k_7\k_8 +  \k_1\k_3\k_5\k_8 = \k_2\k_4\k_6\k_8 \nonumber \\ 
  & \Longleftrightarrow \qquad     \k_1\k_3\k_7+  \k_1\k_3\k_5 = \k_2\k_4\k_6. \label{eq:G2node}
\end{align}
By \eqref{eq:G2node}, if $(\mN, G_2, \kappa)$  is node balanced,  then $\k_1\k_3\k_7- \k_2\k_4\k_6<0$, while 
by \eqref{eq:G1node}, if $(\mN, G_1, \kappa)$  is node balanced,  then $\k_1\k_3\k_7- \k_2\k_4\k_6>0$. This shows that there does not exist $\k$ such that $\mN$ is node balanced with respect to both reaction graphs.
\end{example}

 \section{Horn and Jackson's symmetry conditions}
\label{sec:HJ}

Horn and Jackson studied steady states in general and complex and detailed balanced steady states in particular \cite{hornjackson}. They showed that there exist certain symmetry conditions on the \emph{rate matrix} such that the detailed and complex balanced steady states, and the steady states  in general,  are precisely the points fulfilling these symmetry conditions. We revisit these symmetry conditions in the light of the theory developed here.

Given a kinetics $v(x)$, the $m\times m$ \emph{rate matrix} $\rho(x)$ is such that the $(i,j)$-th entry is  $v_k(x)$ if the $k$-th reaction is $r_k\colon y_j\to y_i$, and zero otherwise. Say a function $\Omega$ on $\R^{m\times m}$  is \emph{symmetric at $x$} if
\begin{equation}\label{eq:symmetry}
\Omega(\rho(x))=\Omega(\rho^T\!(x)),
\end{equation}
where $\rho^T$ denotes the transpose matrix of $\rho$. Define the functions
$$\Omega_d(\rho(x))=\rho(x),\qquad \Omega_c(\rho(x))=\rho(x) e,\qquad \Omega_e(\rho(x))=Y_c\rho(x) e$$
where $e=(1,\ldots,1)\in\R^m$ and $Y_c$ is the labeling matrix of the canonical complex reaction graph. Then any detailed balanced steady state $x^*$ of $\mN$ is characterized by the symmetry condition  for $\Omega_d$, any complex balanced steady state $x^*$ of $\mN$ is characterized by the symmetry condition for  $\Omega_c$ and finally any steady state $x^*$ of $\mN$ is characterized by the symmetry condition for  $\Omega_e$ \cite{hornjackson}. 
We might straightforwardly extend this to node balanced steady states in the following way.

Given a reaction graph $G$, 
we define a function $\iota_G\colon \R^{m\times m}\to \R^{m_G\times m_G}$  entry-wise by $\iota_G(z)_{i_1,i_2}=z_{j_1,j_2}$ if $i_2\to i_1\in E_G$ and $R_{i_2\to i_1}=y_{j_2}\to y_{j_1}$, and $0$ otherwise.
Then $\iota_G(\rho(x))$ is  simply  the rate matrix in terms of the nodes of the reaction graph $G$. Now define
\begin{equation}\label{eq:omegaG}
\Omega_G(\rho(x))=\iota_G(\rho(x)) e_G\,\,\in\,\, \R^{m_G},
\end{equation}
where $e_G=(1,\ldots,1)\in\R^{m_G}$. The $i$-th entry  of the vector in \eqref{eq:omegaG} is the sum of the rate functions of the edges  with target $i$ in $G$. 
Using \eqref{eq:node-condition},  
 note that
 $$\Omega_G(\rho(x))-\Omega_G(\rho^T\!(x))=C_G v(x),$$
which is zero if and only if $x$ is a node balanced steady state with respect to $G$, and if and only if  the symmetry condition is fulfilled for $\Omega_G$ at $x$.

If $G$ is the canonical complex  reaction graph, then $\iota_G$ is the identity map and 
$\Omega_G(\rho(x))=\rho(x)e=\Omega_c(\rho(x)).$ If $G$ is a detailed reaction graph, then there is only one element in each row of $\iota_G(\rho(x))$. Hence the vector $\iota_G(\rho(x)) e_G$, up to a permutation of the entries, agrees with $\Omega_d(\rho(x))$. 

Horn and Jackson speculated that there would be other types of steady states fulfilling similar symmetry conditions as they verified for detailed and complex balanced steady state \cite{hornjackson}. The analysis above confirms that this is indeed the case  and that the function $\Omega_G$ in \eqref{eq:omegaG} perhaps is the natural function to study in this context.

\section{Steady states  of subnetworks}\label{sec:subgraphs}
\label{sec:part}

In this section we present an application of our theory of reaction graphs and node balanced steady states to determine 
whether a node balanced steady state of a network is also node balanced for a subnetwork, and \emph{vice versa}.

Given a reaction network $\mN=(\mS,\mC,\mR)$, the network generated by a subset of reactions $\mR'\subseteq \mR$ is called a \emph{subnetwork} of $\mN$.  Any kinetics of $\mN$ naturally induces a kinetics of  a  subnetwork.

Let $\mN_1,\dots,\mN_\ell$ be subnetworks generated by disjoint subsets of reactions
$\mR_1,\dots,\mR_\ell\subseteq \mR$,  respectively. Let $\mR_{\ell+1} = \mR \setminus \bigcup_{i=1}^\ell \mR_i$ and let $\mN_{\ell+1}$ be the subnetwork generated by $\mR_{\ell+1}$, called the \emph{complementary subnetwork}.
Observe that $\mR_1,\dots,\mR_{\ell+1}$ form a partition of $\mR$.

Consider a reaction graph $G$ associated with $\mN$. For $i=1,\dots,\ell+1$, let $G_i = ( V_i,E_i)$ be the subgraph induced by the edges corresponding to the reactions of $\mR_i$, and let $m_{G_i}$ denote the number of nodes of $G_i$. After renumbering the nodes by means of a bijection  between $\{1,\dots,m_{G_i}\}$ and $V_i$, the graph $G_i$ becomes a reaction graph associated with $\mN_i$, denoted by $G_{\mN_i}$. We say that $G_{\mN_i}$ is a \emph{reaction graph associated with $\mN_i$ induced by $G$}. That is, $G_i$ and $G_{\mN_i}$ are isomorphic digraphs, but $G_i$ is a subgraph of $G$, while $G_{\mN_i}$ is a reaction graph associated with $\mN_i$.

One might construct a new reaction graph $G'$ by taking  the disjoint union of $G_1,\dots,G_{\ell+1}$, or equivalently, the disjoint union of $G_{\mN_1},\dots,G_{\mN_{\ell+1}}$, up to a numbering of the nodes. Specifically, let $\mP$ be the partition defining $[G]$. We define a new partition $\mP'$ as follows:
$j    \sim_{\mP'}  k$ if and only if $j    \sim_{\mP}  k$ and further,  the edge involving $j$ and the edge involving $k$ in $G_{sp}$ correspond to reactions of the same subset  $\mR_i$ (equivalently, the corresponding edges of $G$ belong to the same subgraph $G_i$).

Clearly, $\mP'\leq \mP$ and any reaction graph $G'$  with partition $\mP'$ fulfills $G\preceq G'$. Any such $G'$ is called a reaction graph \emph{induced by $G$ and the subnetworks $\mN_1,\dots,\mN_\ell$.}

Let $n_i$ be the number of species of $\mN_i$ and $\pi_i$ denote the projection from $\R^n$ to $\R^{n_i}$,   obtained by selecting the indices of the species in $\mN_i$.

\begin{proposition}\label{prop:decomposition} 
Consider a reaction network $\mN$ and subnetworks $\mN_1,\dots,\mN_\ell$ with complementary subnetwork $\mN_{\ell+1}$. Let $G$ be a reaction graph, $G_{\mN_i}$   the reaction graph associated with $\mN_i$ induced by $G$, for all $i=1,\dots,\ell+1$, and $G'$ a reaction graph  induced by $G$ and the subnetworks $\mN_1,\dots,\mN_\ell$.
Assume $\mN$ is equipped with mass-action kinetics  for some choice of reaction rate constants. 

Let $x^*\in \R^n_{\geq 0}$. The following statements are equivalent:
\begin{enumerate}[(i)]
\item $x^*$ is a node balanced steady state of $\mN$ with respect to $G$,
and $\pi_i(x^*)$ is a node balanced steady state of $\mN_i$ with respect to $G_{\mN_i}$ for all $i=1,\dots,\ell$. 
\item $x^*$ is a node balanced steady state of $\mN$  with respect to $G'$.
\item $\pi_i(x^*)$ is a node balanced steady state of $\mN_i$ with respect to $G_{\mN_i}$ for all $i=1,\dots,\ell+1$. 
\end{enumerate}
\end{proposition}
\begin{proof}
After renumbering the elements of the set of reactions,  there is a choice of $G''$ with $[G'']=[G']$ such that $C_{G''}$ is a block matrix with $\ell+1$ blocks. Since the statements are  independent of the choice of representative of the equivalence class of $G'$, we assume $G'=G''$. 

The $i$-th block of $C_{G'}$ is precisely the incidence matrix of $G_{\mN_i}$, for the induced numberings. Let $v$ be the kinetics of $\mN$ and  $v_i$ the kinetics induced on $\mN_i$.  Let $p_i$ be the cardinality of $\mR_i$, and let $\rho$ denote the projection from $\R^p$ to $\R^{p_i}$, obtained by selecting the indices of the reactions in $\mR_i$. Note that
$$ v_{i}(\pi_i(x^*)) = \rho_i( v(x^*)). $$

\medskip
(ii) $\Leftrightarrow$ (iii): 
We have  $x^*$ is node balanced with respect to $G'$ if and only if $$C_{ G_{\mN_i}} v_{i}(\pi_i(x^*)) =C_{ G_{\mN_i}} \rho_i( v(x^*))=0,\qquad \textrm{ for all }i=1,\dots,\ell+1,$$ if and only if 
$\pi_i(x^*)$ is a node balanced steady state of $\mN_i$ with respect to $G_{\mN_i}$ for all $i=1,\dots,\ell+1$. 

\medskip
(iii) $\Rightarrow$ (i):  If (iii) holds, then clearly the second part of (i) holds. Further, since we have proven (ii) $\Leftrightarrow$ (iii), $x^*$ 
is a node balanced steady state of $\mN$  with respect to $G'$, which, since $G\preceq G'$, also is a node balanced steady state of $\mN$ with respect to $G$. This shows that  (iii) implies  the first part of (i).

 \medskip
(i) $\Rightarrow$ (ii): 
If (i) holds, then $C_{G}v(x^*)=0$ and 
$$C_{G'} v(x^*)= \left( 0,\dots, 0,  C_{ G_{\mN_{\ell+1}}} v_{\ell+1}(\pi_{\ell+1}(x^*))\right),$$ 
where the zeroes cover the first $m_1+\dots+m_\ell$ entries.  
By Proposition~\ref{prop:incidence}(iii), there exists an $(m_{G}\times m_{G'})$-matrix $B$ such that $C_{G}=BC_{G'}$.
Write $B=(B_1 | \dots | B_{\ell+1})$ such that each block $B_i$ has $m_i$ columns.
Then, by hypothesis we have
\begin{equation}\label{eq:BC}
B_{\ell+1} C_{ G_{\mN_{\ell+1}}}  v_{\ell+1}(\pi_{\ell+1}(x^*))=0.
\end{equation}
Let $\varphi\colon V_{G'}\rightarrow V_G$ be the morphism defining the inclusion $G\preceq G'$.
By construction (see the proof of Proposition~\ref{prop:incidence}(iii)),   the $j$-th column of $B$ has one non-zero entry, equal to one, in the $\varphi(j)$-th row.
By definition of $G'$, $\varphi$ is injective on each $V_{G_i}$. In particular, $B_{\ell+1}$ has $m_{\ell+1}$ non-zero rows, which define a permutation matrix. Thus, \eqref{eq:BC} implies $C_{ G_{\mN_{\ell+1}}}  v_{\ell+1}(\pi_{\ell+1}(x^*))=0$. Hence $C_{G'}v(x^*)=0$.
\end{proof}

Note that if $\delta_G=\delta_{G'}$, then any node balanced steady state with respect to $G$ automatically fulfills the three equivalent statements in Proposition~\ref{prop:decomposition}.

An interesting consequence of Proposition 6 and Theorem 1 is that, given a reaction graph $G$, either all or none of the positive steady states of $\mN$ are  node balanced with respect to $G$ and fulfil that $\pi_i(x^*)$ is a node balanced steady state of $\mN_i$ with respect to $G_{\mN_i}$ for $i=1,\ldots,\ell$. Indeed, either all or none the positive steady states of $\mN$ are node balanced with respect to $G'$.

In the particular setting of complex balanced steady states, the reaction graphs  induced by a complex reaction graph are complex reaction graphs associated with the subnetworks.
Proposition~\ref{prop:decomposition} says that the set of complex balanced steady states that are also complex balanced  for a set of disjoint subnetworks agree with the set of node balanced steady states for a specific reaction graph defined from the subnetworks. Therefore, the properties of node balanced steady states derived from Theorem~\ref{thm:complex}, \ref{prop:k} and \ref{thm:inclusion} also apply in this case. In particular, positive steady states of this type can only exist if the reaction networks $\mN_1,\dots,\mN_\ell$, as well as the complementary subnetwork, are weakly reversible.  Moreover, conditions on the reaction rate constants $\kappa$ for which such steady states exist are given in Theorem~\ref{prop:k}. Finally, if there exists a positive complex balanced steady state of $\mN$ that is also complex balanced for a subnetwork, then so is any  positive steady state of $\mN$.

\begin{myexampleHorn}
Consider the following subsets of $\mR$:
\begin{align*}
\mR_1& =\{r_1,r_2,r_6\},  & \mR_2 & =\{r_3,r_4,r_5\}.
\end{align*}
Let respectively $\mN_1,\mN_2$ denote the subnetworks generated by $\mR_1,\mR_2$.
Consider the complex balanced reaction graph $G_1$ in Figure~\ref{fig:horn}. Then $G_3$ is a reaction graph induced by $G_1$ and $\mN_1$.
By Proposition~\ref{prop:decomposition}, $x^*$ is a complex balanced steady state for $\mN$ and $\mN_1$, if and only if $x^*$ is a node balanced steady state with respect to $G_3$. 
In this case, it is also a complex balanced steady state for $\mN_2$.
\end{myexampleHorn}

\begin{myexampleHorn}
 Let $\mN'$ be the subnetwork generated by $\mR'=\{r_1,r_2,r_5,r_6\}$. Both $\mN$ and $\mN'$ are weakly reversible, but  the complementary subnetwork, with reactions $r_3,r_4$, is not.
 Thus there does not exist reaction rate constants for which there exists a positive complex balanced steady state for both $\mN$ and $\mN'$.
 \end{myexampleHorn}

If the sets of reactions $\mR_1,\dots,\mR_\ell$ are not disjoint, then there is not a general unambiguous answer similar to that of Proposition~\ref{prop:decomposition}.

\section{Proofs}\label{sec:proof}
 
\subsection{Proof of Theorems~\ref{thm:complex} and ~\ref{prop:k}}
To ease the notation, we use $[k] = \{1,\dots,k\}$ for $k\in\N$.

  One way to prove Theorem~\ref{thm:complex} and ~\ref{prop:k}  would be to reproduce the arguments of the original results for complex balanced steady states. Indeed, the original arguments work line by line because 
  it is not explicitly used in the proofs that the complexes (node labels of the reaction graph) are different from each other.
This is not even stated as a requirement in \cite{Dickenstein:2011p1112,Craciun-Sturmfels}. The reader familiar with these results will readily see that this is the case.  However, we take a different route here.

For a given network $\mN$, we construct another network $\mN'$, such that  their steady states agree. Further, the complex balanced steady states of $\mN'$ are in one-to-one correspondence with the node balanced steady states of $\mN$. Hence, we can lift the (known) results for complex balanced steady states such that they hold for node balanced steady states as well.
 
Let $G$ be a reaction graph associated with $\mN$. We start with the construction of the network $\mN'=(\mS',\mC',\mR')$. 
  To this end, we define the  species set  as $\mS'=\mS\times V_G=\{(X_i,j)\st X_i\in \mS, \, j\in V_G\}$, and define 
\begin{equation}\label{eq:yj}
\epsilon=1+\max_{y\in\mC}\sum_{i=1}^ny_i, \quad\text{and}\quad y^j=    \sum_{i=1}^n(Y_j)_i\, (X_i,j).
\end{equation}
Since $Y_j$ is the label of the node $j$, then $(Y_j)_i\in \N$ is the coefficient of $Y_j$ in species $X_i$ in the original network.  The reaction set   $\mR'$ is defined as
  \begin{align}
  \mR' &  = \{ y^i\to y^j \st i\to j\in E_G \} 
    \cup \{\epsilon(X_i,j)\to \epsilon(X_i,j')   \st  (X_i,j),(X_i,j')\in \mS' \},  \label{eq:reaction-epsilon} 
  \end{align}
and the   complex set as
 $$\mC' = \{y^1,\dots,y^{m_G}\} \cup \big\{  \epsilon(X_i,j) \st \  (X_i,j) \in \mS'\big\}. $$
 This set has cardinality $m_G+ n\, m_G=m_G(n+1)$. 
We number the set $\mC'$ according to the order
\begin{equation}\label{eq:complexes-number}
y^1,\dots,y^{m_G}, \epsilon(X_1,1),\dots,\epsilon(X_1,m_G), \dots,\epsilon(X_n,1),\dots,\epsilon(X_n,m_G),
\end{equation}
and the species set analogously: $(X_1,1),\dots,(X_1,m_G), \dots, (X_n,1),\dots, (X_n,m_G).$
Reactions are numbered such that the reactions $1,\dots,p$ correspond to the reactions  in $\mR$ (first set in \eqref{eq:reaction-epsilon}, and the rest of the reactions are ordered increasingly in $i$ and arbitrarily within the subsets of reactions involving $(X_i,j)$, $j=1,\dots,m_G$.
The complex reaction graph $G'_c$ of $\mN'$ will refer to this numbering. 

There is a graph isomorphism   from $G$ and the subgraph of $G'_c$ induced by the nodes $1,\dots,m_G$ (those with label $y^1,\dots,y^{m_G}$) that maps $i\in V_G$ to $i$. The coefficient $\epsilon$ ensures  that  the source and target of  the reactions in \eqref{eq:reaction-epsilon} are not  one of $y^j$ in \eqref{eq:yj}.  Thus the reaction graph $G'_c$ has $n$ extra connected components, one for each $X_i$, whose  nodes are labeled by  $ \epsilon(X_i,j)$, $j=1,\dots,m_G$. These components are complete digraphs since there is a directed edge from every node to every other node.
 
 For example, for the reaction graph
 $$\red{1}, X_1+X_2 \rightarrow \red{2}, 3X_1,$$
 we have $\epsilon=4$ and  the reaction network $\mN'$ consists of the reactions 
 $$(X_1,1) + (X_2,1) \rightarrow 3(X_1,2),\qquad 4(X_1,1) \cee{<=>} 4(X_1,2),\qquad 4(X_2,1) \cee{<=>}   4(X_2,2). $$

 The  species set has $n\, m_G$ elements. We identify $\R^{n\,m_G}$ with $\R^{n\times m_G}$ and index 
   the concentration of the species $(X_i,j)$ as $x_{ij}$.
   Consider the linear map $\pi \colon \R^{n\times m_G}  \rightarrow  \R^n$  and  the injective linear map $g\colon \R^{n}\to\R^{n\times m_G}$ defined by
 \begin{align*}
 \pi(x)_i & =  \sum_{j=1}^{m_G} x_{ij}, &&  \   i\in [n], \qquad  \textrm{for all }x\in  \R^{n\times m_G}, \\
 g(x)_{ij} & =x_i, && \   i\in [n], j\in [m_G], \qquad  \textrm{for all }x\in  \R^{n}.
 \end{align*}
 Note that $g$ maps positive vectors to positive vectors and that $\pi(g(x))=m_G\, x$.

\begin{lemma}\label{lem:defNG} 
Let $s=\dim(S)$ and let $S'$ be the stoichiometric subspace  of  $\mN'$.   Then $S=\pi(S')$. 
Further, $S'$  has dimension $s+n(m_G-1)$ and the deficiency of the network $\mN'$  is $\delta_G$.
\end{lemma}
\begin{proof}
Let $e_{i,j\rightarrow j'}=(X_i,j') - (X_i,j)$ be the vector in $\R^{n\times m_G}$ that has two nonzero entries, the $(i,j)$-th, where it is equal to $-1$, and the $(i,j')$-th where it is equal to $1$. 
Then
$$ S' = \big\langle y^j-y^i \st i\rightarrow j\in E_G \big\rangle + \big\langle e_{i,j\rightarrow j'} \st i\in [n], j,j'\in [m_G]\big\rangle.$$

Let us show that $\pi(S')=S$.   Observe that $\pi(S') \subseteq S$ since
$\pi(e_{i,j\rightarrow j'}) = 0$ and $\pi(y^j-y^i)= Y_j-Y_i.$
Furthermore, if $u\in S$, then 
$$u= \sum_{i\rightarrow j\in E_G} \lambda_{ij} (Y_j-Y_i) = \sum_{i\rightarrow j\in E_G} \lambda_{ij} \pi(y^j-y^i) \,\,\in \,\,\pi(S'),$$
hence $\pi(S')=S$.

To determine $\dim S'$ we do the following. For $x\in \ker \pi$  we have
$ \sum_{j=1}^{m_G} x_{ij} =0$   for all $i\in [n]$.
Consider the vector
$$ x' = \sum_{i=1}^n \sum_{j=2}^{m_G}  x_{ij} e_{i,1\rightarrow j}.$$
Since the nonzero entries of  $e_{i,1\rightarrow j}$ are $(e_{i,1\rightarrow j})_{i,j}=1$ and $(e_{i,1\rightarrow j})_{i,1}=-1$, we have
$x'_{k,k'}=x_{k,k'}$ if $k'\neq 1$. Further, using that  $x\in \ker \pi$ by assumption, we have
$x'_{k1} =   - \sum_{j=2}^{m_G}  x_{kj}  = x_{k1}.$
It follows that $x=x'$ and thus $ \ker \pi = \langle e_{i,1\rightarrow j} \st i\in [n], j\in [m_G]\rangle \subseteq S'$. 
For each $i\in [n]$, the vectors $e_{i,1\rightarrow j}$ are the columns of the incidence matrix of 
a connected graph with $m_G$ nodes, which has rank $m_G-1$. 
It follows that $\dim \ker \pi = n(m_G-1)$. Now, by the first isomorphism theorem
$$ \dim S' =  \dim \ker \pi  + \dim \pi(S') = s+ n(m_G-1).$$
Since $C'$ has $\ell_G+ n$ connected components and $m_G(n+1)$ nodes, the deficiency of $\mN'$ is
$$ \delta_{\mN'} = m_G(n+1) - s - n(m_G-1) - \ell_G-  n=  m_G -s  - \ell_G  =\delta_G. $$
\end{proof}

 \paragraph{Proof of Theorem~\ref{thm:complex}.} 
 In order to prove Theorem~\ref{thm:complex} we use that the theorem holds for complex balanced steady states \cite[Theorem 6A]{hornjackson}.
 Since a positive node balanced steady state is in particular complex balanced, any positive  node balanced steady state is asymptotically stable. Further, if there is one positive  node balanced steady state  with respect to $G$, then the network $\mN$ admits exactly one positive steady state within every stoichiometric compatibility class, which is complex balanced. 
Therefore, to prove the theorem all we need is to show that if there is one positive  node balanced steady state  with respect to $G$, then all positive  steady states are node balanced with respect to $G$.

   We endow $\mN'$ with mass-action kinetics, such that the reaction  rate constant $\k_{i\rightarrow j}$ of $y^i\to y^j$ is that of $Y_i\to Y_j$ for any $i\to j\in E_G$. 
 This implies that the reaction  rate constants of the reaction $R_{i\rightarrow j}$  in $G$ and $G'_c$ agree.
  The  reaction  rate constant of   $\epsilon(X_i,j)\to \epsilon(X_i,j')$ is set to $\k_{i,j\rightarrow j'}=1.$ 
  
By \eqref{eq:yj}, for any complex of the form $y^k\in \mC'$, the only nonzero entries are $y^k_{ik} =(Y_k)_{i}$, $i=1,\dots,n$. 
This gives 
\begin{align}\label{eq:Yg}
g(x)^{y^k} &=  \prod_{(i,j)\in [n]\times [m_G]} \hspace{-0.3cm} g(x)_{ij}^{y^k_{ij}}   =    \prod_{i\in  [n]} x_{i}^{(Y_k)_{i}}  = x^{Y_k}, &
 g(x)^{\epsilon(X_{i},k)} &= x_i^\epsilon. 
\end{align}
Denote the kinetics of $\mN$ by $v$ and that of $\mN'$ by $v'$. Then, by the choice of reaction  rate constants we have,
\begin{equation}\label{eq:vG}
v'(g(x))= (v(x), (x_1^\epsilon, \dots, x_1^\epsilon), \dots, (x_n^\epsilon ,\dots,x_n^\epsilon)),
\end{equation}
where for all $i\in [n]$  the vector $(x_i^\epsilon, \dots, x_i^\epsilon)$ has length $m_G(m_G-1)$.

  The incidence matrix of the canonical complex reaction graph of $\mN'$, denoted by $C'$,  is block diagonal  with the first block equal to $C_G$, and the remaining blocks $2,\dots,n+1$ equal to the incidence matrix of a complete digraph with $m_G$ nodes.  We have the following key lemma.

  \begin{lemma}\label{lem:node-complex-steady-state}
 \begin{enumerate}[(i)]
 \item $C' \, v'(g(x))=  (C_G \, v(x) ,  0,\dots, 0) \,\in\, \R^{m_G+n\, m_G}.$
 \item $x\in \R^n_{>0}$ is a positive node balanced steady state of $\mN$ with respect to $G$ if and only if $g(x)\in \R_{>0}^{n\times m_G}$ is a positive complex balanced steady state for $\mN'$.  
 \item If $x' \in  \R_{>0}^{n\times m_G}$ is a positive complex balanced steady state of $\mN'$, then there exists a unique
$x\in  \R^n_{>0}$   such that $x'=g(x)$. 
\end{enumerate}
 \end{lemma}
\begin{proof}
(i) By the block form of $C'$ and \eqref{eq:vG}, 
every column in the  blocks $2,\dots,n+1$ appear with $+$ and with $-$ sign in the same block. Thus multiplication of $(i+1)$-th block with $(x_i^\epsilon ,\dots,x_i^\epsilon)$ is zero. The result follows now because the first block of $C'$ is $C_G$ and the first $m_G$ entries of $v'(g(x))$ agree with $v(x)$. 

(ii) follows directly from (i).
 
 (iii)  For fixed $i\in [n]$,  consider the incidence matrix  of the complete digraph corresponding to the nodes $(i,j)$ for all $j\in [m_G]$.  
Each row of the matrix has exactly $m_G-1$ entries equal to $-1$ and  $m_G-1$ entries equal to one, since there are $m_G-1$ edges with target $j$ and  $m_G-1$ edges with source $j$.  For the edges with source $j$, the rate function is $(x')^{\epsilon(X_{i},j)} = (x')_{ij}^\epsilon$. Thus complex balancing implies
$$ (m_G-1 ) (x')_{ij}^\epsilon = \sum_{j'\in [m_G], j'\neq j}  (x')_{ij'}^\epsilon,\qquad \textrm{for all }j\in [m_G]. $$
For distinct $j_1,j_2\in [m_G]$ this gives 
$$(m_G-1) ((x')_{ij_1}^\epsilon - (x')_{ij_2}^\epsilon)  = 
\sum_{j'\in V_G, j'\neq j_1}  (x')_{ij'}^\epsilon   - \sum_{j'\in V_G, j'\neq j_2}  (x')_{ij'}^\epsilon
=  (x')_{ij_2}^\epsilon   -  (x')_{ij_1}^\epsilon.   $$
This gives $0= m_G( (x')_{ij_1}^\epsilon - (x')_{ij_2}^\epsilon)$. Since $x'$ is positive, we obtain 
$x'_{ij_1} =x'_{ij_2}$ for all pairs $j_1,j_2$ and all $i\in [n]$.  As a consequence, $x'$ is in the image of $g$.
Unicity follows because $g$ is injective.
\end{proof}

  By Lemma~\ref{lem:node-complex-steady-state}(ii), 
 if $\mN$ admits a positive node balanced steady state  $x^*$ with respect to $G$, then $g(x^*)$ is a positive complex balanced steady state of $\mN'$ and  it follows that all positive steady states of $\mN'$ are complex balanced.  Let now $x^{**}$ be another positive steady state of $\mN$. The stoichiometric compatibility class of $\mN'$ containing $g(x^{**})$ has one positive steady state $x'$, which is complex balanced. By 
 Lemma~\ref{lem:node-complex-steady-state}(iii), there exists $x'' \in \R^n_{>0}$ such that $x'=g(x'')$, and by Lemma~\ref{lem:node-complex-steady-state}(ii), it follows that $x''$ is a node balanced steady state with respect to $G$. Let us show that $x''=x^{**}$. 
 We have that $g(x''-x^{**})\in S'$, since $x'$ and $g(x^{**})$ belong to the same stoichiometric compatibility class of $\mN'$. 
  By  Lemma~\ref{lem:defNG}, it follows that $x''-x^{**} = \frac{1}{m_G} \pi( g(x'-x^{**}) )\in S$. Thus $x''$ and $x^{**}$ are in the same stoichiometric compatibility class of $\mN$ and they are both positive steady states. Since there is a unique positive (complex balanced) steady state in each class, they must coincide.
   This concludes the proof of Theorem~\ref{thm:complex}.
\qed

 \paragraph{Proof of Theorem~\ref{prop:k}.} 
 Consider the labeling matrix $Y'$ of the canonical complex reaction graph of $\mN'$
 with numbering of complexes given in \eqref{eq:complexes-number}. 
We let a vector $u\in   \R^{m_G+n\, m_G}$  be indexed as 
 $$u=(u_1,\dots,u_{m_G},u_{11},\dots,u_{1m_G},\dots,u_{n1},\dots,u_{nm_G})$$
 and use the same indexing for the columns of $Y'$.
The matrix  $Y'$ has one or two nonzero entries in each row: if $(X_i,j)$ is part of  $y^j$, then the row corresponding to this species has two nonzero entries: $(Y_j)_i$ in column $j$ and $\epsilon$ in column $(i,j)$.
If $(X_i,j)$ is not part of $y^j$, then the row corresponding to this species has one nonzero entry:  $\epsilon$ in column $(i,j)$.

  Let $\Psi'$ be the map \eqref{eq:chiG} for the canonical complex reaction graph of $\mN'$.
 We start by giving an explicit isomorphism between $ \ker \Psi'$ and $\ker \Psi_G $, which exists since both vector subspaces have dimension $\delta_G$  (cf. Lemma~\ref{lem:defNG}, Proposition~\ref{lem:defkernel}). Consider the projection map
$$ \rho \colon \R^{m_G + n m_G } \rightarrow \R^{m_G}, \qquad \rho(x)=(x_1,\dots,x_{m_G}).$$
 
 \begin{lemma}
 The linear map $\rho$ induces an isomorphism from $ \ker \Psi'$ to $ \ker \Psi_G$.
 \end{lemma}
 \begin{proof}
We first show that $\rho(x)\in \ker \Psi_G$ for $x\in \ker \Psi'$. Using the form of the labeling matrix  $Y'$,
we have that  if $x\in \ker \Psi'$,    then
\begin{equation}\label{eq:kerpsirelation}
(Y_j)_i \, x_j= - \epsilon x_{ij},\qquad\textrm{ for all }i\in [n], j\in [m_G].
\end{equation} 
Thus, for every $i\in [n]$ we have
$$\sum_{j=1}^{m_G} (Y_j)_i x_j = - \epsilon \sum_{j=1}^{m_G} x_{ij} =0,  $$
by definition of $\Psi'$, since the nodes $(i,1),\dots,(i,m_G)$  form a connected component of $C'$. 
By the correspondence between connected components of $C'$ and $G$, $\sum_{i\in G_k} x_i=0$ for all $k\in [\ell_G]$.  
This shows that $\rho(x)\in \ker \Psi_G$.

Let us find the kernel of $\rho$ restricted to $\ker \Psi'$. 
We have $\rho(x)=0$ if and only if $x_j=0$ for all $j=1,\dots,m_G$. By \eqref{eq:kerpsirelation}, 
it follows that also $x_{ij}=0$ for all $i$ and $j$. As a consequence $x=0$. 
Therefore $\rho$ is an injective linear map between two vector spaces of the same dimension, $\ker \Psi'$ and $\ker \Psi_{\mN_G}$, and it is thus an isomorphism.
\end{proof}

 By Theorems 7 and 9 in \cite{Craciun-Sturmfels}, there exists a positive complex balanced steady state for $\mN'$ with a vector of reaction   rate constants $\k'$  if and only if Theorem~\ref{prop:k}(i) holds, that is,
$$(K')^u=1,\qquad \textrm{for all }u\in  \ker \Psi',$$
where $K'$ is computed from the spanning trees in $C'$ with labels given by  $\k'$.
For the particular choice of reaction  rate constants we have made, 
\begin{equation}\label{eq:Kprimeequal}
K'_{ij}=K'_{ij'} \qquad \textrm{ for all }i\in [n],j, j'\in [m_G].
\end{equation}
Indeed, the reaction  rate constants of any spanning tree rooted at $\epsilon(X_i,j)$ or $\epsilon(X_i,j')$ (in the corresponding  connected  component) are equal to one. Moreover $\epsilon(X_i,j)$ and $\epsilon(X_i,j')$ are in the same connected component, which is a complete digraph. Hence the number of spanning trees rooted at $\epsilon(X_i,j)$ and $\epsilon(X_i,j')$ is the same.

Further, using that $G$ and the subgraph of $C'$  with nodes $1,\dots,m_G$ are isomorphic and preserve the reaction  rate constants, it holds that
\begin{equation}\label{eq:KG2}
K_{G} = \rho(K'). 
\end{equation}
If $u\in\ker \Psi'$ , then $\sum_{j=1}^{m_G} u_{ij} =0$. Therefore
$$(K')^u= \rho(K')^{\rho(u)}\prod_{i=1}^{n}\prod_{j=1}^{m_G} (K'_{ij})^{u_{ij}} =
\rho(K')^{\rho(u)}\prod_{i=1}^{n} (K'_{i1})^{\sum_{j=1}^{m_G}u_{ij}} = \rho(K')^{\rho(u)} =K_{G}^{\rho(u)}.$$ 
Since $\rho$ is an isomorphism between $\ker \Psi'$ and $\ker \Psi_G$, we conclude that 
\begin{equation}
(K')^u=1,\qquad \textrm{for all }u\in  \ker \Psi', \qquad \Leftrightarrow \qquad K_{G}^u=1, \qquad \textrm{for all }u\in  \ker \Psi_G.
\end{equation}
 Consequently,  we have  $K_{G}^u=1$ for all $u\in  \ker \Psi_G$  if and only if $\mN'$ admits a positive complex balanced steady state for a choice of reaction  rate constants $\k'$ such that $\k'_{i,j\rightarrow j'}=1$  for all $i\in [n]$, $j,j'\in [m_G]$ and $\k'_{i\rightarrow j} = \k_{i\rightarrow j}$ for all $i\rightarrow j\in E_{C'}$, if and only if $\mN$ admits  a positive node balanced steady state with respect to $G$
for the corresponding choice of  reaction  rate constants $\k$ (Lemma~\ref{lem:node-complex-steady-state}(ii)).  This proves (i).

\medskip
The proof of part (ii) follows the same line of arguments.  We use the same notation for reaction  rate constants $\k$ of $\mN$ and $\k'$ of $\mN'$. By \cite[Eq (21)]{Dickenstein:2011p1112},   (ii) holds for complex balanced steady states of $\mN'$.  That is,
$x' \in \R^{n\times m_G}_{>0}$ is a complex balanced steady state of $\mN'$  if and only if 
\begin{align*}
K'_i (x')^{y^i} -  K'_{j}  (x')^{y^j} & =0, \quad \forall(i,j)\in E_G, \qquad  and \\
K'_{ij} \, (x')^{\epsilon(X_i,j)} -  K'_{ij'} \,(x')^{\epsilon(X_i,j')} & =0,\qquad \forall i\in [n], j,j'\in [m_G].
\end{align*}
Let $x'=g(x)$.  By   \eqref{eq:Yg} and \eqref{eq:Kprimeequal}, the equations in the second row are satisfied.
According to \eqref{eq:Yg} and \eqref{eq:KG2}, the equations in the first row are  equal to
$ K_i x^{Y_i} -  K_{j}  x^{Y_j}=0$. 
This gives the following:
$x$ is a positive node balanced steady state  of $\mN$ with respect to $G$, 
if and only if $g(x)$ is a positive complex balanced steady state   of $\mN'$ (Lemma~\ref{lem:node-complex-steady-state}(ii)), if and only if $ K_i x^{Y_i} -  K_{j}  x^{Y_j}=0$ for all $(i,j)\in E_G$.
Thus (ii) holds and the theorem is proven.
\qed

\subsection{Proof of Theorem~\ref{thm:inclusion}}

To simplify the notation, we let $K$ denote $K_G$ and $K'$ denote $K_{G'}$ throughout the subsection.

\paragraph{(i) The nodes $i_1,i_2$ do \emph{not} belong to the same connected component of $G$. }
In this case, $\delta_G=\delta_{G'}$ by Proposition \ref{prop:deficiency}. We will show that  the two sets of equations in $\k$ for which $(\mN,G,\k)$ and $(\mN,G',\k)$ are node balanced  can be chosen to be the same. 
Let $G_{p_1}$ and $G_{p_2}$ be the connected components of $G$ containing $i_1,i_2$ respectively, and $G'_q$ be the connected component of $G'$ containing   $k$.
The morphism $\varphi$  is a bijection between $V_{G}\setminus \{i_1,i_2\}$ and $V'_G\setminus \{k\}$ and induces an isomorphism between the subgraph of $G$ obtained by removing $G_{p_1},G_{p_2}$ and the subgraph of $G'$ obtained by removing $G'_q$.
Let $V_i$ (resp. $V_i'$) denote the set of nodes of the $i$-th connected component of $G$ (resp. $G'$).

 We define  two linear maps
$\alpha\colon \R^{m_G} \rightarrow \R^{m'_G}$, and $\beta\colon \R^{m'_G} \rightarrow \R^{m_G}$
by
$$ \alpha(x)_i = \begin{cases}
x_{\varphi^{-1}(i)} & \textrm{if }i\neq k \\
x_{i_1} + x_{i_2} & \textrm{if }i = k,
\end{cases}\qquad  
\beta(x)_i = \begin{cases}
x_{\varphi(i)} & \textrm{if }i\neq i_1,i_2 \\
\sum_{j\in V_{p_2}}x_{\varphi(j)} & \textrm{if }i = i_1 \\
\sum_{j\in V_{p_1} }x_{\varphi(j)} & \textrm{if }i = i_2. 
\end{cases}$$

\begin{proposition}\label{eq:keronecomp}
The morphisms $\alpha$ and $\beta$ induce isomorphisms between $\ker \Psi_G$ and $\ker \Psi_{G'}$. 
\end{proposition}
\begin{proof}
We first show that $\alpha\circ \beta (x)=x$ if $x\in \ker \Psi_{G'}$.
For $i\neq k$, we have $\varphi^{-1}(i)\neq i_1,i_2$ and this gives
$$\alpha \circ\beta(x)_i  = \beta(x)_{\varphi^{-1}(i)}  = x_{\varphi\circ \varphi^{-1}(i)} = x_i.$$
For $i=k$,   using that $x\in \ker \Psi_{G'}$, we have
$$ \alpha \circ\beta(x)_k =\beta(x)_{i_1} + \beta(x)_{i_2} = \sum_{j\in V_{p_2}}x_{\varphi(j)} + 
\sum_{j\in V_{p_1}}x_{\varphi(j)}  = x_k + \sum_{j\in V'_q} x_j = x_k.$$
 This proves that $ \alpha \circ\beta(x)=x$. In particular, $\beta$ is injective.

 We show now that
  $\beta(\ker \Psi_{G'})\subseteq \ker \Psi_{G}$. 
Recall that  $Y'_{\varphi(i)}= Y_i$ for all $i\in [m_G]$.
Consider the labeling matrices $Y,Y'$ and the linear maps  induced by $Y,Y'$ in $\R^{m_G}$ and $\R^{m_G'}$, respectively. We will show that 
$Y = Y'\circ \alpha$ in $\R^{m_G}$.
For $x\in \R^{m_G}$  we have
\begin{align*}
Y'\circ\alpha(x)  & = \sum_{i=1}^{m'_G} \alpha(x)_i  Y'_{i} =
(x_{i_1} + x_{i_2}) Y'_{k} + \sum_{i=1, i\neq k}^{m'_G} x_{\varphi^{-1}(i)} Y'_{i}   \\ &=   
x_{i_1}Y_{i_1} + x_{i_2} Y_{i_2} + \sum_{j=1, j\neq i_1,i_2}^{m_G'} x_{j} Y_{\varphi(j)} 
= \sum_{j=1}^{m_G} x_j Y_j = Y x.
\end{align*}
Let  $x\in \ker \Psi_{G'}$. 
For $j\neq p_1,p_2$, let $G'_{t(j)}$ be the connected component of $G'$ isomorphic to $G_j$ by $\varphi$. 
Then 
\begin{align*}
Y \circ\beta(x) & = Y'\circ \alpha \circ\beta (x) = Y'x =0 \\
 \sum_{i\in  V_j}   \beta(x)_i  & = \sum_{i\in  V_j} x_{\varphi(i)}= \sum_{i\in  V'_{t(j)}} x_{i}=0,\qquad j\neq p_1,p_2, \\
  \sum_{i\in  V_{p_1}}   \beta(x)_i &= \beta(x)_{i_1} +  \sum_{i\in  V_{p_1}, i\neq i_1} x_{\varphi(i)} =
\sum_{i\in  V_{p_2}} x_{\varphi(i)} +  \sum_{i\in  V_{p_1}, i\neq i_1} x_{\varphi(i)} =
\sum_{i\in V'_q} x_i=0.
\end{align*}
The last equation with the roles of $p_1$ and $p_2$ interchanged holds analogously.
Thus $\Psi_G(\beta(x))=0$ and $\beta(x)\in \ker \Psi_{G}$.
Since $\dim \ker \Psi_{G'} = \delta_{G'} = \delta_G = \dim \ker \Psi_{G}$ and $\beta$ is injective,
$\beta$ is an isomorphism and so is $\alpha$.  In particular $\alpha(\ker \Psi_{G})\subseteq \ker \Psi_{G'}$.
\end{proof}

The proposition shows  how the morphism $\beta$ can be used to find a basis of $\ker \Psi_{G}$ from a basis of  $\ker \Psi_{G'}$, provided that $G'$ is obtained from $G$ by joining one pair of nodes  and $G'$ has one fewer connected components than $G$.  

\begin{proposition}\label{prop:Kagree} 
Let $u\in \ker \Psi_{G'}$. Then
$$ {K'}^{u} = K^{\beta(u)}. $$
\end{proposition}
\begin{proof}
If $i$ is a node of $G$ that does not belong to $G_{p_1}$ nor $G_{p_2}$, then
$K_{i} = K'_{\varphi(i)}.$
If $i\in V_{p_1}$, then any spanning tree rooted at $\varphi(i)$ in $G'_q$ consists of the image by $\varphi$ of the union of a spanning tree rooted at $i$ in $G_{p_1}$ and one tree rooted at $i_2$ in $G_{p_2}$. 
Thus
$K'_{\varphi(i)} = K_{i}K_{i_2}$.  If $i \in V_{p_2}$, then we obtain analogously that 
$K'_{\varphi(i)} = K_{i}K_{i_1}.$ Note that if $i=i_1$ or $i_2$ and $\varphi(i)=k$, then the two equations agree, $K'_{k} = K_{i_1}K_{i_2}$.
Using the definition of $\beta$, this gives:
$$ (K')_{\varphi(i)}^{u_{\varphi(i)}} = \begin{cases} 
K_i^{\beta(u)_i} & \textrm{if }i\notin V_{p_1} \cup V_{p_2} \\
K_i^{\beta(u)_i} K_{i_2}^{u_{\varphi(i)}}  & \textrm{if }i\in V_{p_1}  \setminus \{i_1\} \\
K_i^{\beta(u)_i} K_{i_1}^{u_{\varphi(i)}}  & \textrm{if }i\in V_{p_2}  \setminus \{i_2\} \\
K_{i_1} ^{u_{\varphi(i_1)} } K_{i_2} ^{u_{\varphi(i_2)} } & \textrm{if }i=i_1,i_2.
\end{cases} $$
Thus
\begin{align*}
{K'}^{u}   & = \prod_{j=1}^{m_{G'}} (K')_{j}^{u_j} = (K')_k^{u_k}\hspace{-0.2cm} \prod_{i\in [m_G]\setminus \{i_1,i_2\}} \hspace{-0.2cm} (K')_{\varphi(i)}^{u_{\varphi(i)}} 
= \mathopen{\raisebox{-0.5ex}{$\Biggl[$}}\prod_{i\in [m_G]\setminus \{i_1,i_2\}} \hspace{-0.3cm} K_i^{\beta(u)_i}   \mathopen{\raisebox{-0.5ex}{$\Biggl]$}}  \mathopen{\raisebox{-0.5ex}{$\Biggl[$}}  \prod_{i\in V_{p_1} }  K_{i_2}^{u_{\varphi(i)}}  \prod_{i\in V_{p_2} }  K_{i_1}^{u_{\varphi(i)}}   \mathopen{\raisebox{-0.5ex}{$\Biggl]$}}  \\ 
&= \mathopen{\raisebox{-0.5ex}{$\Biggl[$}} \prod_{i\in [m_G]\setminus \{i_1,i_2\}}    \hspace{-0.2cm} K_i^{\beta(u)_i}     \mathopen{\raisebox{-0.5ex}{$\Biggl]$}}
 K_{i_2}^{\beta(u)_{i_2} } K_{i_1}^{\beta(u)_{i_1}} = K^{\beta(u)}. 
  \end{align*}
  This concludes the proof.
\end{proof}

We are ready to prove Theorem~\ref{thm:inclusion}(i). Using Theorem~\ref{prop:k}(i), $(\mN,G',\k)$ is node balanced if and only if 
${K'}^u=1$ for all $u\in \ker \Psi_{G'}$. By Proposition~\ref{prop:Kagree}, this is equivalent to $K^{\beta(u)}=1$ for all  $u\in \ker \Psi_{G'}$, which in turn is equivalent to $K^u=1$ for all $u\in \ker \Psi_{G}$, because $\beta$ is an isomorphism. Using Theorem~\ref{prop:k}(i), the later condition is equivalent to $(\mN,G,\k)$ being node balanced. 
 This concludes the proof of Theorem~\ref{thm:inclusion}(i).
 \qed

\paragraph{Case 2: The nodes $i_1,i_2$ belong to the same connected component of $G$.}  In this situation, we have  $\delta_{G'} = \delta_{G}-1$ and $m_G=m_G'+1$. 
Let  $G_p$ be the connected component of $G$ containing $i_1,i_2$, and $G'_q$  the connected component of $G'$ that contains $k$. 
Again, we have  that $\varphi$ induces an isomorphism between $G$ and $G'$ outside the connected components $G_p$ and $G'_q$.

Consider the linear and injective morphism
$ \gamma \colon \R^{m'_G} \rightarrow \R^{m_{G}}$
defined for $i\in [m_G]$ by
$$\gamma(x)_{i} =\begin{cases}
x_{\varphi(i)} & \textrm{if } i\neq i_2 \\
0 & \textrm{if } i =i_2.
\end{cases} $$
(All  constructions could be done alternatively by replacing $i_2$ with $i_1$).

\begin{proposition}\label{prop:basis_ker}
Let $\delta(i_1,i_2)\in \R^{m_G}$ be the vector with  $\delta(i_1,i_2)_{i_1} = 1$, $\delta(i_1,i_2)_{i_2}=-1$ and the rest of entries equal to zero.
Let $\{u_1,\dots,u_{\delta_{G'}}\}$ be a basis of $ \ker \Psi_{G'}$.
Then
$$\big\{ \gamma(u_1),\dots, \gamma(u_{\delta_{G'}}),\delta(i_1,i_2) \big\}$$
is a basis of $ \ker \Psi_{G}$.
\end{proposition}
\begin{proof}
Since the $i_2$-th component of $\gamma(u_j)$ is zero for all $j\in [\delta_{G'}]$, the vector $\delta(i_1,i_2)$ does not belong to $\langle  \gamma(u_1),\dots, \gamma(u_{\delta_G}) \rangle$. Since $\gamma$ is injective, the vectors $\gamma(u_1),\dots, \gamma(u_{\delta_{G'}}),\delta(i_1,i_2)$ are linearly independent and generate a vector space of dimension $\delta_{G'}+1 = \delta_G=\dim \ker \Psi_{G}$. Thus all we need is to show that  $\delta(i_1,i_2),\gamma(u)\in \ker \Psi_{G}$  for all $u\in  \ker \Psi_{G'}$.

We have
$$ Y \delta(i_1,i_2) = Y_{i_1}  - Y_{i_2} = Y'_{k} - Y'_{k} =0. $$
 Since $i_1$ and $i_2$ belong to the same connected component of $G$, we have
$\sum_{i\in G_j} \delta(i_1,i_2)_i=0$ for all connected components $G_j$ of $G$. Thus $\delta(i_1,i_2)\in \ker \Psi_{G}$.
 
For a connected component $G_j$ of $G$, let $G'_{t(j)}$ be the corresponding connected component of $G'$ under the morphism $\varphi$, such that $t(p)=q$. For $u\in  \ker \Psi_{G'}$ and $j\in [\ell_G]$, we have
\begin{align*}
 Y \gamma(u) & = \sum_{i=1}^{m_G} \gamma(u)_i  Y_i =
\sum_{i=1, i\neq i_2}^{m_G}   u_{\varphi(i)} Y_i  = \sum_{i=1, i\neq i_2}^{m_G}   u_{\varphi(i)} Y'_{\varphi(i)} 
=  \sum_{j=1}^{m_{G'}}   u_{j} Y'_{j} = Y' u =0, \\ 
\sum_{i\in G_j} \gamma(u)_i &=   \sum_{i\in G_j, i\neq i_2} u_{\varphi(i)} =   \sum_{i\in G'_{t(j)}} u_{i} =0.
\end{align*}
This shows that $\gamma(u)\in \ker \Psi_{G}$, which concludes the proof.
\end{proof}

\begin{proposition}\label{prop:formal}
Assuming
$K^{\delta(i_1,i_2)} =1,$ then for all $u\in \ker \Psi_{G'}$ it holds
$$ K^{\gamma(u)} ={K'}^{u}.$$
\end{proposition}
\begin{proof}
Because $\gamma$ induces an isomorphism of digraphs outside $G_p$ and $G'_q$,  we can without loss of generality  restrict  the proof to the case, where $G,G'$ are strongly connected. For simplicity, we 
 let $m=m_{G'}$ such that $V_G=[m+1]$ and $V_{G'}=[m]$, and assume that $\varphi$ is the identity on $\{1,\dots,m\}$ and sends $m+1$ to $1$ (so $i_1=1, i_2=m+1, k=1$). In this setting, $K^{\delta(i_1,i_2)} =1$ is equivalent to $K_1=K_{m+1}$.
 
Assume that it holds 
\begin{equation}\label{eq:K1}
K_{1} K'_{i} - K_{i} K'_{1}=0,\qquad \textrm{for all }i=1,\dots,m.
\end{equation}
Not that since the graphs are strongly connected, none of $K_i,K_i'$ is zero.
 Then, using $K_{1} = K_{m+1}$, we have 
 \begin{align*}
K^{\gamma(u)} &= \prod_{i=1}^{m} K_{i}^{u_i}=
 \prod_{i=1}^{m}  K_{1}^{u_i} \, \frac{K_{i}^{u_i} }{K_{1}^{u_i}} 
 = \left[\prod_{i=1}^{m} K_{1}^{u_i}\right]  \prod_{i=1}^{m}   \left[   \frac{K'_{i}}{K'_{1}}\right]^{u_i}
  = \left[\frac{K_{1}}{K_1'}\right]^{\sum_{i=1}^{m} u_i}  \prod_{i=1}^{m}  (K')_{i}^{u_i} = (K')^u, 
  \end{align*}
where we use $\sum_{i=1}^{m} u_i=0$ since $u\in \ker \Psi_{G'}$.

Thus, all we need is to show that  \eqref{eq:K1} holds provided $K_1=K_{m+1}$. 
Below we use the indices $k,\ell$ generically.
Let $\Theta'_j$ be the set of spanning trees of $G'$ rooted at $j$. Given $F,B\subseteq [m+1]$ of cardinality $M$, let $\Theta(F,B)$ be the set of spanning forests of $G$ with $M$ connected components, such that each component is a tree rooted at one element of $B$ and contains exactly one element of $F$. 
Let 
$$\Theta_{i,j} = \Theta(\{1,m+1\},\{i,j\}). $$
If $\zeta\in \Theta_{i,m+1}$, then $\zeta$ is a spanning forest that consists of two trees $\zeta_i,\zeta_{m+1}$ rooted respectively at $i$ and $m+1$ such that $1$ is a node of $\zeta_i$. Analogously, $\zeta\in \Theta_{1,i} $ is a spanning forest that consists of two trees $\zeta_1,\zeta_{i}$  rooted respectively at $1$ and $i$   such that $m+1$ is a node of $\zeta_i$.
 Then,  $\varphi$ induces two one-to-one correspondences
\begin{align*}
\Theta'_i & \longleftrightarrow  \Theta_{i,m+1}  \sqcup \Theta_{1,i}  &
\Theta'_1 & \longleftrightarrow   \Theta_{1,m+1},
\end{align*}
for $2\leq i\leq m$.
The one-to-one correspondence between $E_G$ and $E_{G'}$ induces a bijective map $\widehat{\varphi}$ from the set of subgraphs of $G$ to the set of subgraphs of $G'$. It is thus enough to show that $\widehat{\varphi}$ maps $ \Theta_{i,m+1}  \sqcup \Theta_{1,i}$  to $\Theta'_i$ and 
$  \Theta_{1,m+1}$ to $\Theta'_1$. First, observe that for any subgraph $\zeta$ of $G$, $\widehat{\varphi}(\zeta)$   contains an undirected cycle if and only if $\zeta$ contains an undirected cycle or an undirected path joining 1 and $m+1$. Hence, for $\zeta$  in $\Theta_{i,m+1}$, $\Theta_{1,i}$ or $\Theta_{1,m+1}$,  $\widehat{\varphi}(\zeta)$ is  acyclic. Moreover, $\widehat{\varphi}(\zeta)$ is connected because $\zeta$ consists of two disjoint trees, one containing 1 and the other containing $m+1$. Hence, $\widehat{\varphi}(\zeta)$ is a tree, and since $\zeta$ is a spanning forest, $\widehat{\varphi}(\zeta)$ is a spanning tree. Finally, if $\zeta\in\Theta_{i,m+1}  \sqcup \Theta_{1,i}$, then $\zeta$ has no edge with source $i$, so $\widehat{\varphi}(\zeta)\in\Theta'_i$. Similarly, $\zeta\in\Theta_{1,m+1}$ has no edge with source $1$ or $m+1$, so $\widehat{\varphi}(\zeta)\in\Theta'_1$, as desired.

By hypothesis, $K_1=K_{m+1}$. Thus
\begin{equation}\label{eq:rewriteK1}
K_1 K_i' = K_1 \left[ \sum_{\zeta\in  \Theta_{i,m+1}} \zeta^\k + \sum_{\zeta\in  \Theta_{1,i}} \zeta^\k \right] =
K_1 \left[ \sum_{\zeta\in  \Theta_{i,m+1}} \zeta^\k\right]+ K_{m+1} \left[ \sum_{\zeta\in  \Theta_{1,i}} \zeta^\k \right].
\end{equation}

We make now use of the All-Minors Matrix-Tree theorem, which extends equation \eqref{eq:Kminor}. 
Since we are assuming $G$ is connected, the Laplacian matrix $L$ is of size $(m+1)\times (m+1)$ (we omit reference to $\k$ for simplicity). Recall that $L_{(j)}$ is the minor of $L$ obtained by removing the row $m+1$ and columns $j$ of $L$ and taking the determinant. 
Let $L_{(j,k)}$ for $j<k$ be the minor of $L$ obtained by removing rows $1,m+1$ and columns $j,k$. 
Then it holds  that \cite[Th 3.1]{Moon94} 
$$L_{(j,k)} = (-1)^{j + k +1} \sum_{\zeta\in  \Theta_{j,k}} \zeta^\k. $$
In view of \eqref{eq:rewriteK1} and \eqref{eq:Kminor}, to prove \eqref{eq:K1} we need to show that
$$(-1)^{1+i+m+1+1}  L_{(1)} L_{(i,m+1)} + (-1)^{m+1+i+1+1}  L_{(m+1)} L_{(1,i)} = (-1)^{i+1+m+1+1}  L_{(i)} L_{(1,m+1)},$$
that is, we have to show that it holds
\begin{equation}\label{eq:32}
L_{(1)} L_{(i,m+1)} +   L_{(m+1)} L_{(1,i)} =  L_{(i)} L_{(1,m+1)}.
\end{equation}

We find $L_{(1)}$, $L_{(m+1)}$ and $L_{(i)}$ by expanding the corresponding submatrix of the Laplacian used for the computation of each of the minors along the first row. Then \eqref{eq:32} is equivalent to
\begin{equation}\label{eq:minors1}
	\sum_{j=2}^{m+1} (-1)^{j} L_{(1,j)} L_{(i,m+1)} +   \sum_{j=1}^{m} (-1)^{j+1} L_{(j,m+1)}  L_{(1,i)} =  \sum_{j=1,j\neq i }^{m+1} (-1)^{\epsilon_j} L_{(i,j)}   L_{(1,m+1)}.
\end{equation}
where $\epsilon_j=j+1$ if $j<i$ and $j$ if $j>i$.
 
Equation~\eqref{eq:minors1} has three  sums, which we refer to as the first, second and third sum for simplicity, reading from left to right.
The summand for $j=i$ in the first two sums is respectively $(-1)^i L_{(1,i)} L_{(i,m+1)}$ and $(-1)^{i+1} L_{(i,m+1)}  L_{(1,i)}$, which cancel out. 
The two terms for $j=1$ in the second and third sums agree, using that $\epsilon_1=2$. Finally, the summands for $j=m+1$ in the first and the third sums agree since  $\epsilon_{m+1}=m+1$.
 
 We consider now the three summands, one for each  sum, corresponding to a fixed $j\in \{2,\dots,m\}$, $j\neq i$. 
We will use the Pl\"ucker relations on the maximal minors of a rectangular matrix \cite{Kleiman:1972kl}. These  are as follows. 
 Consider a $d\times n$ matrix  $A$, $d<n$. For a set $I\subseteq [n]$, let $A_{I}$ be the minor obtained after removing the columns with index in $I$ of $A$. Consider two sets $I,K$ of cardinality $n-d+1$ and $n-d-1$, respectively. Then it holds that
 $$\sum_{k\in I\setminus K} (-1)^{\bar{\epsilon}_k}  A_{I\setminus k} A_{K\cup \{k\}}=0,$$
 with $\bar{\epsilon}_k= \# \{\ell \in I \st k<\ell\} +  \# \{\ell \in K \st k<\ell\}$.
 (See for example \cite{Kleiman:1972kl}. Note that the formula given in \cite{Kleiman:1972kl} is stated in complementary notation,  indicating the columns that are kept to construct the minor, and not those that are removed).
 
 Let $A$ be the $(m-1)\times (m+1)$ submatrix of $L$ obtained by removing the first and last rows. 
  We apply the Pl\"ucker relation to $A$ with the sets $I=\{1,j,m+1\}$, $K=\{i\}$ and obtain 
  $$ (-1)^{\bar{\epsilon}_{m+1}} L_{(1,j)} L_{(i,m+1)} + (-1)^{\bar{\epsilon}_j} L_{(1,m+1)} L_{(i,j)} + 
  (-1)^{\bar{\epsilon}_1} L_{(j,m+1)} L_{(1,i)}  =0. $$ 
  We have $\bar{\epsilon}_{m+1}=0$,  $\bar{\epsilon}_{1}= 2+1=3,$ and $\bar{\epsilon}_{j} = 1$ if $j>i$, and $\bar{\epsilon}_{j} = 2$ if $j<i$. 
  Thus, after rearranging the terms, we obtain
  $$(-1)^{j} L_{(1,j)} L_{(i,m+1)} +   (-1)^{j+1} L_{(j,m+1)}  L_{(1,i)} -   (-1)^{\epsilon_j} L_{(j,m+1)}   L_{(1,m+1)}=0.$$ 
  This implies that \eqref{eq:minors1}  holds and concludes the proof.  
\end{proof}

We are now ready to prove Theorem~\ref{thm:inclusion}(ii). By Theorem~\ref{prop:k}(i) and Proposition~\ref{prop:basis_ker}, 
$(\mN,G,\k)$ is node balanced if and only if $K^{\gamma(u)}=1$ for all $u\in \ker \Psi_{G'}$ and $K^{\delta(i_1,i_2)}=1$. By Proposition~\ref{prop:formal}, this is equivalent to ${K'}^{u}=1$ for all $u\in \ker \Psi_{G'}$ and $K_{i_1}=K_{i_2}$, which in turn is equivalent to 
$(\mN,G',\k)$ being node balanced and $K_{i_1}=K_{i_2}$, by Theorem~\ref{prop:k}(i).
 This concludes the proof.
 \qed

\medskip

\subsection*{Acknowledgements} EF and CW acknowledge funding from the Danish Research Council for Independent Research. CW is supported by  Dr.phil. Ragna Rask-Nielsen Grundforskningsfond (administered by the Royal Danish Academy of Sciences and Letters). Part of this work was done while EF and CW visited Universitat Polit\`ecnica de Catalunya (Barcelona) in Spring-Summer 2017.


 \end{document}